\documentclass[11pt, english]{article}

\usepackage[margin= 2 cm,bottom=19mm, top=18mm]{geometry}

\usepackage[square,sort,comma,numbers]{natbib}
\usepackage{amsthm}
\usepackage{amsmath}
\usepackage{amssymb}
\usepackage{setspace}
\usepackage{mathtools}
\usepackage{graphicx}
\graphicspath{ {./images/} }
\usepackage[hidelinks]{hyperref}
\usepackage{cleveref}
\usepackage{hyperref}
\usepackage{enumitem}
\usepackage{framed}
\usepackage{subcaption}
\usepackage{xspace}
\usepackage{thmtools} 
\usepackage{thm-restate}

\usepackage{thmtools}
\usepackage{thm-restate}

\usepackage{floatrow}

\usepackage{tikz}
\usepackage{mathdots}
\usepackage{xcolor}
\usepackage{diagbox}
\usepackage{colortbl}
\usepackage[absolute,overlay]{textpos}

\usetikzlibrary{calc}
\usetikzlibrary{decorations.pathreplacing}
\usetikzlibrary{positioning,patterns}
\usetikzlibrary{arrows,shapes,positioning}
\usetikzlibrary{decorations.markings}

\tikzstyle{edge}=[very thick]
\definecolor{bostonuniversityred}{rgb}{0.8, 0.0, 0.0}
\definecolor{arsenic}{rgb}{0.23, 0.27, 0.29}
\tikzstyle{diredge}=[postaction={decorate,decoration={markings,
		mark=at position .95 with {\arrow[scale = 1]{stealth};}}}]
\tikzset{
    arrow/.style={decoration={markings, mark=at position 0.7 with
    {\fill(-0.09*#1,-0.03*#1) -- (0,0) -- (-0.09*#1,0.03*#1) -- cycle;}}, postaction={decorate}},
    arrow/.default=1
}
\tikzset{
    arow/.style={decoration={markings, mark=at position 1 with
    {\fill(-0.09*#1,-0.03*#1) -- (0,0) -- (-0.09*#1,0.03*#1) -- cycle;}}, postaction={decorate}},
    arow/.default=1
}
\tikzset{
    arrrow/.style={decoration={markings, mark=at position 0.9 with
    {\fill(-0.09*#1,-0.03*#1) -- (0,0) -- (-0.09*#1,0.03*#1) -- cycle;}}, postaction={decorate}},
    arow/.default=1
}

\newcommand{\fitellipsis}[2] 
{\draw [fill=white]let \p1=(#1), \p2=(#2), \n1={atan2(\y2-\y1,\x2-\x1)}, \n2={veclen(\y2-\y1,\x2-\x1)}
    in ($ (\p1)!0.5!(\p2) $) ellipse [ x radius=\n2/2+0cm, y radius=1.1cm, rotate=\n1];
}
\newcommand{\Fitellipsis}[2] 
{\draw [fill=white]let \p1=(#1), \p2=(#2), \n1={atan2(\y2-\y1,\x2-\x1)}, \n2={veclen(\y2-\y1,\x2-\x1)}
    in ($ (\p1)!0.5!(\p2) $) ellipse [ x radius=\n2/2+0cm, y radius=1.4cm, rotate=\n1];
}

\theoremstyle{plain}

\newtheorem*{thm*}{Theorem}
\newtheorem{thm}{Theorem}[section]
\Crefname{thm}{Theorem}{Theorems}

\newtheorem*{lem*}{Lemma}
\newtheorem{lem}[thm]{Lemma}
\Crefname{lem}{Lemma}{Lemmas}

\newtheorem*{claim*}{Claim}
\newtheorem{claim}[thm]{Claim}
\crefname{claim}{Claim}{Claims}
\Crefname{claim}{Claim}{Claims}

\newtheorem{prop}[thm]{Proposition}
\Crefname{prop}{Proposition}{Propositions}

\newtheorem{remar}[thm]{Remark}
\Crefname{remar}{Remark}{Remarks}

\crefname{cor}{Corollary}{Corollaries}

\crefname{conj}{Conjecture}{Conjectures}

\Crefname{qn}{Question}{Questions}

\Crefname{obs}{Observation}{Observations}

\Crefname{ex}{Example}{Examples}

\theoremstyle{definition}

\Crefname{prob}{Problem}{Problems}

\newtheorem{defn}[thm]{Definition}
\Crefname{defn}{Definition}{Definitions}

\theoremstyle{remark}

\renewenvironment{proof}[1][]{\begin{trivlist}
\item[\hspace{\labelsep}{\bf\noindent Proof#1.\/}] }{\qed\end{trivlist}}

\newcommand{\remove}[1]{}

\newcommand{\eps}{\varepsilon}

\title{\vspace{-1 cm}
Tight bounds for powers of Hamilton cycles in tournaments}
\date{}

\author{
Nemanja Dragani\'{c} \thanks{
Department of Mathematics, ETH, Z\"urich, Switzerland. Research supported in part by SNSF grant 200021\_196965.
\newline
\emph{E-mails}: \textbf{\{nemanja.draganic,david.munhacanascorreia,benjamin.sudakov\}@math.ethz.ch}.
}
\and
David Munh\'a Correia\footnotemark[1]
\and
Benny Sudakov\footnotemark[1]}

\begin{document} 
\maketitle
\begin{abstract}
A basic result in graph theory says that any $n$-vertex tournament with in- and out-degrees larger than $\frac{n-2}{4}$ contains a Hamilton cycle, and this is tight. In 1990, Bollob\'{a}s and H\"{a}ggkvist significantly extended this by showing that for any fixed $k$ and $\eps > 0$, and sufficiently large $n$, all tournaments with degrees at least $\frac{n}{4}+\eps n$ contain the $k$-th power of a Hamilton cycle. Up until now, there has not been any progress on determining a more accurate error term in the degree condition, neither in understanding how large $n$ should be in the Bollob\'{a}s-H\"{a}ggkvist theorem. We essentially resolve both of these questions. First, we show that if the degrees are at least $\frac{n}{4} + cn^{1-1/\lceil k/2 \rceil}$ for some constant $c = c(k)$, then the tournament contains the $k$-th power of a Hamilton cycle. In particular, in order to guarantee the square of a Hamilton cycle, one only needs a constant additive term. We also present a construction which, modulo a well-known conjecture on Tur\'an numbers for complete bipartite graphs, shows that the error term must be of order at least $n^{1-1/\lceil (k-1)/2 \rceil}$, which matches our upper bound for all even $k$. For odd $k$, we believe that the lower bound can be improved. Indeed, we show that for $k=3$, there are tournaments with degrees $\frac{n}{4}+\Omega(n^{1/5})$ and no cube of a Hamilton cycle. In addition, our results imply that the Bollob\'{a}s-H\"{a}ggkvist theorem already holds for $n = \eps^{-\Theta(k)}$, which is best possible.
\end{abstract}

\section{Introduction}\label{main}
Hamiltonicity is one of the most central notions in graph theory, and it has been extensively studied by numerous researchers. The problem of deciding Hamiltonicity of a graph is NP-complete, but there are many important results which derive sufficient conditions for this property. One of them is the classical Dirac's theorem \cite{dirac1952some}, which states that every graph with minimum degree at least $\frac{n}{2}$ contains a Hamilton cycle, and that this is tight. Another natural and more general property is to contain the $k$-th power of a Hamilton cycle. Extending Dirac's theorem, Seymour \cite{seymour1974problem} conjectured in 1974
that the minimum degree condition for a graph to contain the $k$-th power of a Hamilton cycle is $\frac{kn}{k+1}$. After two decades and several papers on this question, 
Koml\'{o}s, S\'{a}rk\"{o}zy and Szmer\'{e}di \cite{komlos1998proof} confirmed Seymour's conjecture.

Clearly, one can ask similar questions for directed graphs (see \cite{kuhn2012survey}), which tend to be more difficult. In 1979, Thomassen \cite{thomassen1979long} asked the question of determining the \emph{minimum semidegree} $\delta^0(G)$ (that is, the minimum of all in- and out-degrees) which implies the existence of a Hamilton cycle in an oriented graph $G$. This was only answered thirty years later by Keevash, K\"{u}hn and Osthus \cite{keevash2009exact}, who showed that $\delta^0(G) \geq \frac{3n-4}{8}$ forces a Hamilton cycle, which is tight by a construction of H\"{a}ggkvist \cite{haggkvist1993hamilton}. Already the problem for squares of Hamilton cycles is not well understood. Treglown \cite{treglown2012note} showed that $\delta^0(G) \geq \frac{5n}{12}$ is necessary, which was subsequently improved by DeBiasio (personal communication).  He showed that $\delta^0(G) \geq \frac{3n}{7}-1$ is needed, using a slightly unbalanced blowup of the Paley tournament on seven vertices. It would be interesting to determine, even asymptotically,  the optimal value of $\delta^0(G)$ which implies the existence of the square of Hamilton cycle. For clarity, by the $k$-th power of the directed path $P_l = v_0 \ldots v_l$ we mean the directed graph $P^k_l$ on the same vertex set with an edge $v_iv_j$ if and only if $i < j \leq i+k$. The $k$-th power of a directed cycle is similarly defined. 

Due to the difficulty of these problems in general, it is natural to ask what happens in tournaments. It is a well known result that every tournament with minimum semidegree $\frac{n-2}{4}$ has a Hamilton cycle and that this is best possible. By how much do we need to increase the degrees in order to guarantee a $k$-th power? The remarkable result by Bollob\'{a}s and H\"{a}ggkvist \cite{bollobas1990powers} given below
says that a little bit is already enough. 
\begin{thm}[\cite{bollobas1990powers}]\label{bollobas}
For every $\eps > 0$ and $k$, there exists a $n_0 = n_0(\eps,k)$ such that every tournament $T$ on $n \geq n_0$ vertices with $\delta^0(T) \geq \frac{n}{4} + \eps n$ contains the $k$-th power of a Hamilton cycle.
\end{thm}
\noindent This theorem suggests two very natural avenues of research. The first one is to determine the order of magnitude of the additive error term in the degree condition. Indeed, the proof of the above theorem does not give any additional information, apart from showing that it is $o(n)$. It is a common pattern in extremal combinatorics that once asymptotic bounds are obtained,  they are gradually refined using new methods and inventive ideas, until the right behaviour of the error term is shown. A celebrated example of this line of research is the problem of determining minimal degree conditions in $k$-uniform hypergraphs which force a perfect matching, which has been extensively studied. After a series of improvements of the additive error term in the necessary degree \cite{kuhn2006matchings,aharoni2009perfect, rodl2008approximate,rodl2006dirac}, the precise behaviour was proved in 2009 by R\"odl, Ruci\'nski and Szemer\'edi \cite{rodl2009perfect} using the powerful absorption method. Also, in the case of Thomassen's question mentioned earlier, concerning degree conditions forcing hamiltonicity of oriented graphs, it was first shown in \cite{kelly2008dirac} that a minimum semidegree $\frac{3n}{8} + o(n)$ was sufficient - in a subsequent paper, Keevash, K\"uhn and Osthus \cite{keevash2009exact} confirmed the conjecture that $\frac{3n-4}{8}$ is enough. A similar situation occurred with Seymour's conjecture (see \cite{komlos1998proof} and the references therein). Finally, the classical  Erd\H{o}s-Stone-Simonovits theorem about extremal numbers $\text{ex}(n,H)$ of graphs \cite{erdos1946structure} was also refined over the years, with some results showing that the size of the forbidden graph $H$ can be logarithmic in $n$ \cite{chvatal1983notes, bollobas1994extension, bollobas1976structure, ishigami2002proof} and others obtaining a better error term for the number of edges in the extremal graphs \cite{simonovits1968method}.

The second natural question that arises from Theorem~\ref{bollobas} is to determine, for fixed $\eps$, how large $n_0$ should be as a function of $k$. The proof of Bollob\'as and H\"aggkvist needs $n_0$ to be very large, more precisely, to grow faster than $t(\lceil \log_2 (1/\eps) \rceil + 2)$, where $t$ is a tower-type function defined by letting $t(0) = 2^k$ and $t(i+1) = \frac{1}{2}(2\eps)^{-t(i)}$. Despite that, one might hope for only an exponential upper bound for $n_0$, which would be best possible since a random tournament on at most $2^{(k-1)/2}$ vertices does not even contain a transitive tournament of size $k$, let alone the $k$-th power of a Hamilton cycle.

In this paper we address both these questions, resolving the second one and obtaining nearly tight bounds for the first one. 
We start with the additive error in the degree condition.
\begin{restatable}{thm}{MainTheorem}\label{maintheorem}
There exists a constant $c = c(k) > 0$ such that any tournament $T$ on $n$ vertices with $\delta^{0}(T) \geq \frac{n}{4} + cn^{1- 1/ \lceil k/2 \rceil}$ contains the $k$-th power of a Hamilton cycle.
\end{restatable}
\noindent
In particular, we show that a constant error term is enough for the tournament to contain the square of a Hamilton cycle. 

It appears that this theorem is nearly tight. This follows from a somewhat surprising connection between our question and the Tur\'an problem for complete bipartite graphs.
As usual, for a fixed graph $H$  we let $\text{ex} (n,H)$ denote the maximal number of edges in a $n$-vertex graph which does not contain $H$ as a subgraph. 
The next result gives a construction of tournaments with large minimum semidegree which do not contain the $k$-th power of a Hamilton cycle. 
\begin{thm}\label{generallowerbound}
Let $k \geq 2$ and $r = \lceil \frac{k-1}{2} \rceil$. For all sufficiently large $n = 3  \text{ }(\text{mod } 4)$, there exists a $n$-vertex tournament $T$ with $\delta^{0}(T) \geq \frac{n+1}{4} + \Omega \left(\frac{ex(n,K_{r,r})}{n} \right)$ which does not contain the $k$-th power of a Hamilton cycle.
\end{thm}
Modulo a well-known conjecture on Tur\'an numbers for complete bipartite graphs, this result implies that in addition to $n/4$, the semidegree bound must have an additive term of order at least $n^{1- 1/ \lceil (k-1)/2 \rceil}$, which matches the bound in Theorem \ref{maintheorem} for all even $k$. Indeed, the celebrated result of K\"{o}v\'{a}ri, S\'{o}s and Tur\'{a}n \cite{kHovari1954problem}, says that $\text{ex} (n,K_{r,r}) = O(n^{2- 1/r})$ and this estimate is widely believed to be tight.
Moreover, for unbalanced complete bipartite graphs, it was proven by Alon, Koll\'ar, R\'onyai and Szab\'o (\cite{alon1999norm},\cite{kollar1996norm}) that $\text{ex} (n,K_{r,s}) = \Omega(n^{2- 1/r})$ when $s > (r-1)!$.
It is also known that $\text{ex} (n,K_{r,r}) = \Omega(n^{2- 1/r})$ for $r = 2,3$ (\cite{Erdos1966}, \cite{brown1966graphs}), which corresponds to $k=4,6$ in our problem. 

For odd values of $k$ there is still a small gap between the results in Theorems \ref{maintheorem} and \ref{generallowerbound}, which would be interesting to bridge. We make a step in this direction, showing that for $k=3$, the constant error term in Theorem \ref{generallowerbound} can be improved to a power of $n$. 
\begin{restatable}{thm}{cubes}\label{cubesthm}
For infinitely many values of $n$, there exists a tournament on $n$ vertices with minimum semidegree $\frac{n}{4} + \Omega(n^{1/5})$ and no cube of a Hamilton cycle. 
\end{restatable}
\subsection{Proof outline}\label{proofstrategy}\label{sec:proof overview}
The main idea in the proof of Theorem \ref{maintheorem} is based on a dichotomy that occurs in the structure of tournaments. We say that a tournament $T$ is \emph{$\delta$-cut-dense} if any balanced partition $(X,Y)$ of $V(T)$ is such that $\overrightarrow{e}(X,Y) \geq \delta|X||Y|$. Note in particular, that every tournament with minimum semidegree at least $\frac{n}{4} + \frac{\delta n}{2}$ is $\delta$-cut-dense. 
We will first consider tournaments which are cut-dense and show that they contain the $k$-th power of a Hamilton cycle
even if the minimum semidegree is slightly below $\frac{n}{4}$. After this, we consider a tournament which has a balanced cut that is sparse in one direction. An overview of what we do for each case is given below.

\nopagebreak[4]
\vspace{0.2cm}
\nopagebreak[4]
\noindent
{\bf Cut-dense tournaments}
\nopagebreak[4]
\vspace{0.2cm}

\nopagebreak[4]
\noindent The following theorem deals with the case of cut-dense tournaments and also provides an answer to the first question raised in the introduction. It shows that the Bollob\'as-H\"aggkvist theorem holds already when $n$ is exponential in $k$. 
\begin{restatable}{thm}{Cutresult}\label{thm:cut result}
Let $k \geq 2$, $\delta > 0$ and $n \geq \left( \frac{3}{\delta} \right)^{1000k}$. Then, any $\delta$-cut-dense tournament $T$ such that $\delta^0(T) \geq \frac{n}{4} - \frac{\delta n}{200}$ has the $k$-th power of a Hamilton cycle. 
\end{restatable}
\noindent 
As noted above, tournaments with minimum semidegree at least $\frac{n}{4}+\eps n$ are $2\eps$-cut-dense. Thus, this result implies that we can take $n_0 = \eps^{-O(k)}$ in Theorem \ref{bollobas}. In Section \ref{sec:lower bound cubes}, we show that this behaviour is optimal.  

The first idea in the proof of the above theorem is to partition the tournament into so-called chains $C$, which are ordered structures with the following properties:
\begin{itemize}
\item {\it Robustness.} $C$ is such that even if we delete some of its vertices which are somewhat sparsely distributed in $C$, we get a structure which contains the $k$-th power of a path.
\item {\it Large neighborhoods.} The first $k$ vertices in $C$ have a large common in-neighborhood, and the last $k$ have a large common out-neighborhood.
\end{itemize}
In order to find this partition into chains we use the recent result in \cite{draganic2020powers}, which shows that one can always find the $k$-th power of a long path in a tournament. We apply this iteratively, until a certain constant number of vertices is left. By using the semidegree condition, we also absorb these vertices into other chains.
Call the obtained (disjoint) chains $\mathcal{C}=\{C_1,\ldots,C_t\}$.
To finish the proof, we "link" the chains, by always connecting the last $k$ vertices of $C_i$ to the first $k$ vertices of $C_{i+1}$ (and $C_t$ to $C_1$) with $k$-th powers of paths. In order to create the links between the chains, we are free to use the internal vertices of the other chains in $\mathcal{C}$, but in such a way that the robustness property ensures that after deleting the used vertices from the chains in
$\mathcal{C}$, we still have $k$-th powers of paths. This gives the desired $k$-th power of a Hamilton cycle.

The main tool for linking two chains will be the linking lemma (Lemma \ref{lem:linking}). Suppose we want to link $C_i$ to $C_{i+1}$. The rough idea is to consider the set of vertices $A$ which are in some sense \textit{reachable} by $k$-th powers of paths starting at the set of last $k$-vertices in $C_i$; similarly, we consider the set $B$ of vertices which can reach the first $k$ vertices of $C_{i+1}$. Because of the minimum semidegree condition, $A$ and $B$ will be of sizes close to $n/2$. Then we consider two cases. Either the intersection $S=A\cap B$ is large, and thus we can find a connection between $C_i$ and $C_{i+1}$ which passes through $S$; or $S$ is small, and then we can use the $\delta$-cut-dense property of $T$, to find a connection between $A$ and $B$, and consequently establish a connection between $C_i$ and $C_{i+1}$. 

\vspace{0.2cm}
\noindent
{\bf Tournaments with a sparse cut}
\vspace{0.2cm}

\noindent
In the second part of the proof, we consider the case of $T$ having a balanced cut which is sparse in one of the directions, that is, the number of edges in this direction is $o(n^2)$. The first thing to do is to convert this cut into sets $A,B,R$ which partition the vertex set and are such that $|R| = o(n)$ and both tournaments $T[A],T[B]$ are $o(n)$-almost regular, of size $\frac{n}{2}-o(n)$ and such that $\overrightarrow{e}(A,B) = o(n^2)$. 

Let us first consider the case when $R = \emptyset$ and $|A| = |B| = \frac{n}{2}$ in order to give a rough outline of some ideas. As a preliminary, note that since the average out-degree in $T[A]$ is at most $\frac{n}{4}$, we have that $\overrightarrow{e}(A,B) \geq c|A|n^{1-1/\lceil k/2 \rceil} = \Omega(n^{2-1/\lceil k/2 \rceil})$. Naturally, the first step is to find a way to cross from $A$ to $B$ and from $B$ to $A$. Specifically, we will want to find transitive subtournaments $A_1,A_2 \subseteq A$ and $B_1,B_2 \subseteq B$ of size $k$ such that $(A_1,B_1)$ forms a $k$-th power of a path of size $2k$ starting at $A_1$ and ending at $B_1$ and $(B_2,A_2)$ forms a $k$-th power starting at $B_2$ and ending at $A_2$. Now, since the density from $B$ to $A$ is $1-o(1)$, finding $A_2$ and $B_2$ is not difficult. The bottleneck of the problem is in finding $A_1$ and $B_1$, which is heavily dependent on the number of edges going from $A$ to $B$. 

Indeed, assume that such $A_1$ and $B_1$ exist and define $A'_1$ to be the set of last $\lceil k/2 \rceil$ vertices of $A_1$ and $B'_1$ the set of first $\lceil k/2 \rceil$ vertices of $B_1$. Then, we have that every vertex in $A'_1$ dominates every vertex in $B'_1$, which creates a $K_{\lceil k/2 \rceil,\lceil k/2 \rceil}$ in the graph formed by the edges going from $A$ to $B$. Therefore, in general, to find such sets $A_1$ and $B_1$ we would need
the number of edges from $A$ to $B$ to be at least the Tur\'an number of  $K_{\lceil k/2 \rceil,\lceil k/2 \rceil}$ which is believed to be $\Theta(n^{2-1/\lceil k/2 \rceil})$.
In Section \ref{abridge}, we show that this number of edges is also sufficient to find $A_1, B_1$ as above.

After constructing the sets $A_1,A_2,B_1,B_2$, it is simple to finish. Recall that $T[A]$ and $T[B]$ are $o(n)$-almost regular. Therefore, $T[A \setminus (A_1 \cup A_2)]$ and $T[B \setminus (B_1 \cup B_2)]$ are $\frac{1}{4}$-cut-dense tournaments and so, by Theorem \ref{thm:cut result} they contain spanning chains $C_A$ and $C_B$. We can link the end of the chain $C_A$ to $A_1$, then link $B_1$ to the start of $C_B$, the end of $C_B$ to $B_2$ and finally link $A_2$ to the start of $C_A$. This produces the $k$-th power of a Hamilton cycle.

The general case builds on the above approach. The main goal will be to cover the set $R$ with a collection $\mathcal{B}$ of $o(n)$ vertex-disjoint structures, which we will call bridges. Informally, a bridge is the $k$-th power of a path which intersects $A \cup B$ in at most $4k$ vertices, such that the sets of first $k$ vertices and last $k$ are contained entirely in $A$ or $B$. We will say that a bridge goes from $B$ to $A$, for example, if the first $k$ vertices are in $B$ and the last $k$ are in $A$. We also want these first $k$ to have large common in-neighborhood in $B$ (so that it can be linked later to the rest of $B$) and the last $k$ to have large common out-neighborhood in $A$. Finally, we will crucially need the collection $\mathcal{B}$ to satisfy the following property. The number of bridges going from $A$ to $B$ is positive and equal to the number of them going from $B$ to $A$. Indeed, note that if this is the case, we can then construct the $k$-th power of a Hamilton cycle by using these bridges to cover $R$ and then using the $o(n)$-almost regularity of both $T[A]$ and $T[B]$ to link them to the rest of the vertices in $A$ and $B$, like we did in the case $R=\emptyset$. Since the number of bridges going from $A$ to $B$ is positive, we are able to cross from $A$ to $B$ at least once, and further, since we have the same number of bridges going from $B$ to $A$, we are able to cross the same number of times from $B$ to $A$. The construction of bridges is delicate and requires several ideas which we discuss in Section \ref{sec:proof of main result}.

\vspace{0.5 cm} 
\noindent \textbf{Notation:} Throughout the paper we use standard graph theoretic notation. We use the following notation for directed graphs. An oriented graph is a directed graph in which between any two vertices there is at most one edge.
A tournament is an oriented complete graph. Let $T$ be a tournament. By $V(T)$ we denote its set of vertices. Let $\{v\},X,Y\subseteq V(T)$. By $\overrightarrow{E}(X,Y)$ we denote the set of edges in $T$ oriented from $X$ to $Y$, and we let $\overrightarrow{e}(X,Y)=|\overrightarrow{E}(X,Y)|$. If $\overrightarrow{E}(X,Y)=X\times Y$ we say that $X$ dominates $Y$, and write $X\rightarrow Y$. If $\{v\}$ dominates $X$ we also write $v\rightarrow X$. By $N^-(v)$ we denote the set of \emph{(common) in-neighbors} of $v$, i.e. the set of vertices $u$ such that $u\rightarrow v$; by $N^+(v)$ we denote the set of \emph{(common) out-neighbors} of $v$, i.e. the set of vertices $u$ such that $v\rightarrow u$. Furthermore, we let $N^-_X(v)=N^-(v)\cap X$ and $N^+_X(v)=N^+(v)\cap X$; we let $d_X^-(v)=|N^-_X(v)|$ be the in-degree and $d_X^+(v)=|N^+_X(v)|$ the out-degree of $v$ in $X$.
By $\delta^0(T)$ we denote the \emph{minimum semidegree} of $T$, i.e. the minimum over all in- and out-degrees in $T$. A \emph{balanced cut} of $T$ is a partition $X,Y$ of its vertex set such that $||X|-|Y|| \leq 1$. A tournament $T$ is said to be $d$-almost regular if $\delta^0(T) \geq \frac{1}{2}n - d$. If a tournament is $1$-almost regular, we say it is regular. By $T[X]$ we denote the subtournament in $T$ induced by $X$. If $Z$ is an ordered subset of $V(T)$, we slightly abuse notation by saying that $Z$ is a $k$-th power of a path whenever $T[Z]$ contains a spanning directed $k$-th power of a path where the ordering of the path is inherited from the ordering of $Z$. Similarly, we say that disjoint sets of vertices $X_1,\ldots,X_t$ in $V(T)$ form the \emph{$k$-blowup of a path}, if every $X_i$ is of size $k$, and $X_i\rightarrow X_{i+1}$ for all $i\in [t-1]$; we call the sets $X_i$ \emph{blobs}.
\vspace{0.5 cm}

\noindent 
The rest of this paper is organized as follows. In Section \ref{sec:preliminaries}, we will gather some useful results and prove a few technical lemmas that will be used in our proofs. In Section \ref{cutdense}, we prove Theorem \ref{thm:cut result} and in Section \ref{abridge}, we set the stage for the proof of our main result; in both sections we give results which might be of independent interest. In Section \ref{sec:proof of main result}, we prove Theorem \ref{maintheorem}. In Section \ref{sec:lower bound cubes}, we give the proofs of Theorems \ref{generallowerbound} and \ref{cubesthm}. Finally, we make some concluding remarks in Section \ref{sec:concluding remarks}.
\section{Preliminaries}\label{sec:preliminaries}

\subsection{A few simple lemmas} 
We will first give several standard lemmas and definitions. The first lemma, which we state without proof, is the following folklore result about tournaments.
\begin{lem}\label{lem:transitive subtournament}
Let $T$ be a tournament on at least $2^{k-1}$ vertices. Then it contains a transitive sub-tournament on $k$ vertices. Furthermore, a fraction of more than $1/2^{k^2}$ of its $k$-subsets induce transitive tournaments.
\end{lem}

\noindent
Apart from the standard Ramsey-type result, the above lemma also gives us a density statement for transitive tournaments of size $k$. This follows from the following simple averaging argument, which we also use in 
Lemma \ref{inclusion-exlusion}. Suppose we know that every set $S$ of $s$ vertices contains a $k$-subset which satisfies some property $\mathcal{P}$ (for example, being transitive). We now take a set $R$ on at least $s$ vertices and want to count the number of $k$-subsets of $R$ which satisfy the property $\mathcal{P}$. Since every set of $s$ vertices contains a $k$-set satisfying $\mathcal{P}$ and every $k$-set is contained in ${\binom{|R|-k}{s-k}}$ sets of size $s$, we know that the number of $k$-subsets of $R$ which satisfy property $\mathcal{P}$ is at least ${\binom{|R|}{s}}/{\binom{|R|-k}{s-k}} > \frac{1}{s^k} {|R| \choose k}$. Thus, the fraction of $k$-sets in $R$ satisfying $\mathcal{P}$ is greater than $1/s^k$.

We also state here the following result from  \cite{draganic2020powers} which was mentioned in the introduction.
\begin{thm}\label{thm:long paths}
Every tournament on $n$ vertices contains a copy of $P_m^k$ where $m\geq\frac{n}{2^{8k}}$ and $k>1$.
\end{thm}
Another useful lemma about tournaments is the following, which shows that tournaments which are cut-dense and have high minimum semi-degree are such that all of their cuts are dense, not only the ones that are balanced.

\begin{lem}\label{lem:anycutisdense}
Let $T$ be a $\delta$-cut-dense tournament on $n$ vertices such that $\delta^0(T)\ge \frac{n}{4}-\frac{\delta n}{16}$. If $(X,Y)$ is a partition of its vertex set, then $\overrightarrow{e}(X,Y) \geq \frac{\delta}{16}|X||Y|$.
\end{lem}
\begin{proof}
Suppose without loss of generality that $|X| \leq |Y|$. Then, by double-counting the sum $\sum_{v \in X} d^{+}(v)$ we get
$$|X| \left(\frac{n}{4} - \frac{\delta n}{16}\right) \leq \sum_{v \in X} d^{+}(v) \leq \frac{|X|^2}{2} + \overrightarrow{e}(X,Y)$$
which implies that $\overrightarrow{e}(X,Y) \geq |X| \left(\frac{n}{4} - \frac{\delta n}{16} - \frac{|X|}{2} \right)$. If $|X| \leq \frac{n}{2} - \frac{\delta n}{4}$, then this gives $\overrightarrow{e}(X,Y) \geq \frac{\delta}{16}|X||Y|$. Otherwise, we can obtain $X,Y$ from a balanced partition by moving at most $\frac{\delta n}{4}$ vertices from one part to the other. Hence, since $T$ is $\delta$-cut-dense, we have
$$\overrightarrow{e}(X,Y) \geq \frac{\delta n^2}{4} - \frac{\delta n}{4} \cdot \frac{n}{2} = \frac{\delta n^2}{8} \geq \frac{\delta}{8}|X||Y|$$
\end{proof}
Next, we will give a lemma about powers of paths, showing that if we delete vertices from a large power, if those vertices are in some sense sparsely distributed, then the remaining vertices still span a large power of a path. Before, we will need the following definition, which captures the notion of being sparsely distributed.
\begin{defn}
Let $\pi=(v_0,\ldots,v_m)$ be an ordering of a set of vertices $V$ and let $\mathcal{S} = \{S_1,\ldots,S_q\}$ be a collection of sets. We say that $\mathcal{S}$ is \emph{$r$-apart in $\pi$} if for any two distinct $S_i,S_j$, the distance in $\pi$ between any two  vertices $u_i \in S_i \cap V$ and $u_j \in S_j \cap V$ is more than $r$. For a set $A$, we let $I_r(A,\pi)$ denote the set of vertices $v \in V$ such that there is some $a \in A \cap V$ which is at a distance of at most $r$ from $v$.
\end{defn}
\begin{lem} \label{lem:destroying powers}
Let $r_1 > r_2$ and $T$ be a tournament. Suppose $P = (v_0,v_1\ldots, v_m)$ is the $r_1$-th power of a path. Let $\mathcal{S} = \{S_i\}$ be a collection of subsets of $V(T)$ which is $r_1$-apart in $P$ and such that $|S_i|\leq r_2$ for all $i$. Then, $P \setminus \bigcup_i S_i$ is the $(r_1-r_2)$-th power of a path.
\end{lem}
\begin{proof}
Let $P \setminus \bigcup_i S_i = (v_{j_1},v_{j_2}, \ldots, v_{j_t})$ with $j_1 < j_2 < \ldots < j_t$. Suppose for a contradiction that there is a $t'$ such that $j_{t' + r_1-r_2} > j_{t'} + r_1$. Consider the interval $I = \{v_{j_{t'}}, v_{j_{t'} + 1}, \ldots, v_{j_{t'} + r_1}\}$. Then, there must be more than $r_2$ vertices in $I \cap \bigcup_i S_i $. However, by the assumption on the sets $S_i$ there can be at most one $i$ such that $I \cap S_i \neq \emptyset$, thus implying that $|S_i| > r_2$, which contradicts the assumption.
\end{proof}
We now state the following standard lemma which uses K\"{o}v\'{a}ri-S\'{o}s-Tur\'{a}n-like ideas.
\begin{lem}\label{prop:KST}
Let $\alpha \in (0,1)$ and $t,k, N \in \mathbb{N}$ be such that $t \geq \frac{k}{\alpha}$. Let $\{S_1, \ldots, S_t\}$ be a collection of subsets of $[N]$ of size at least $\alpha N$. Then, there exists $k$ sets $S_i$ which intersect in at least $\left( \frac{\alpha}{e} \right)^k N$ elements.
\end{lem}
\begin{proof}
Let $x=\left( \frac{\alpha}{e} \right)^k N$. Define a bipartite graph $G$ with parts $\{S_1,\ldots, S_t\}$ and $[N]$ and let $\{S_i,j\}$ be an edge if and only if $j\in S_i$. Suppose, for contradiction sake, that every $k$ sets $S_i$ have less than $x$ common neighbors in $[N]$. Then, by double-counting and convexity, note that \begin{equation}\label{EquationGoesBrrrr}
x\binom{t}{k} > \sum_{v\in [N]}\binom{d(v)}{k} \geq  N \binom{\alpha t}{k}
\end{equation}
Note that ${t \choose k}/{\alpha t \choose k} \leq \left(\frac{e t}{k} \right)^k / \left(\frac{\alpha t}{k} \right)^k$;
so, by definition of $x$, we have ${t \choose k}/{\alpha t \choose k} \leq \frac{N}{x}$, contradicting (\ref{EquationGoesBrrrr}). 
\end{proof}
We now give some definitions of structures which will be crucial throughout the paper. After, we give a lemma which is a consequence of Lemmas \ref{lem:transitive subtournament} and \ref{prop:KST}. 
\begin{defn}\label{def:head and tail}
Let $T$ be a tournament and let $A,S\subseteq V(T)$. We say that $S$ is a \emph{$A$-head (resp. $A$-tail)}, if $S$ induces a transitive tournament and is such that $N_A^+(S)\geq |A|/2^{6k}$ (resp. $N_A^-(S)\geq |A|/2^{6k}$), where $k = |S|$. If $A=V(T)$ then we just say that $S$ is a \emph{head (resp. tail)}. Furthermore, if $P\subseteq T$ is the $k$-th power of a path and $A,B \subseteq V(T)$ are such that the first $k$ vertices of $P$ form a $A$-tail and the last $k$ form a $B$-head, then we say that $P$ is a \emph{$(A,B,k)$-chain}. If $A = B = V(T)$, we just say that $P$ is a \emph{$k$-chain}.
\end{defn}
\begin{lem}\label{inclusion-exlusion}
Let $T$ be a tournament and $A,S \subseteq V(T)$ such that $|S| \geq 2^{30k}$. Suppose that every vertex $v \in S$ has at least $|A|/20$ out(resp. in)-neighbors in $A$. Then, $S$ contains a $A$-head (resp. $A$-tail) of size $k$. Consequently, any subset of at least $2^{30k}$ vertices of $T$, all of which have at least $|A|/20$ out(resp. in)-neighbors in $A$ is such that a fraction of at least $1/2^{30k^2}$ of its $k$-sets are $A$-heads (resp. $A$-tails).
\end{lem}
\begin{proof}
By Lemma \ref{lem:transitive subtournament}, $S$ contains a transitive tournament $R$ of size $30k$. For every vertex $v \in R$, define $S_v = N^{+}_A(v) \subseteq A$. Then $|S_v| \geq |A|/20$. Thus, we can apply Lemma \ref{prop:KST} to the collection $\{S_v: v \in R\}$ with $\alpha = \frac{1}{20}$, to find a subset $R' \subseteq R$ of size $k$ such that $N^{+}_A(R') \geq |A|/(20e)^k \geq |A|/2^{6k}$, hence $R'$ is an $A$-head.
\end{proof}
The following lemma is a version of dependent random choice (see, e.g., survey \cite{FS} for more information about this technique and its applications) which we will use several times in the paper.
\begin{lem}\label{DRC for sets}
Let $G$ be a bipartite graph on $(A,B)$ and $\gamma > 0$ be such that $e(A,B)=\gamma |A||B|$. Let $a = |A|,b = |B|$ and $m,k,s \in \mathbb{N}$ be such that 
\[
\binom{a\gamma^k}{k}>2s^k\max \left(\binom{a}{k}\left(\frac{m}{b}\right)^k,1\right).
\]
Then there exists a set $U\subseteq A$ with $|U|\geq s$, such that less than a fraction of $1/s^k$ of the sets of size $k$ in $U$ have less than $m$ common neighbors in $B$.
\end{lem}
\begin{proof}
Let $S$ be a subset of $k$ random vertices, chosen uniformly from $B$ with repetition. Let $U$ denote the set of common neighbors of $S$ in $A$ and $X = |U|$. Let $Y$ denote the number of $k$-subsets of $U$ with less than $m$ common neighbors in $B$. Note that,
$\mathbb{E}[X] \geq \sum_{v\in A} \left(\frac{d(v)}{|B|}\right)^k\geq \gamma^k a$, where the second inequality follows from Jensen's inequality; applying the same inequality also gives $\mathbb{E} \left[{X \choose k} \right] \geq {\mathbb{E}[X] \choose k}$. 
Moreover, we have
$\mathbb{E}[Y] \leq {|A| \choose k} \left(\frac{m}{|B|} \right)^k$.
Now, note that
\begin{equation}\label{comparing expectation}
\mathbb{E} \left[{X \choose k} \right]> s^k (E[Y]+1)
\end{equation}
implies, by linearity of expectation, that there is a choice of $U$ such that $Y$ is at most a $1/s^k$ fraction of the total number of $k$-sets in $U$, and furthermore $X=|U|>s$. Since (\ref{comparing expectation}) follows from the previous bounds and the assumption of the lemma, we are done.
\end{proof}
\subsection{Chain partitioning and linking}\label{technical lemmas}
We will now prove two important lemmas which will be used throughout our proofs. The first lemma is used to partition the vertex set of a tournament into a small number of chains and a small set of remaining vertices.
\begin{lem}\label{lem:partition lemma}
Let $T$ be a tournament and $S,A,B\subseteq V(T)$ be such that every $v\in S$ has $|N^-_A(v)|\geq |A|/20$ and $|N_B^+(v)|\geq |B|/20$. Then we can partition $S$ into at most $2^{200k}\log |S|$ vertex disjoint $(A,B,k)$-chains and a set of remaining vertices of size at most $2^{200k}$.
\end{lem}
\begin{proof}
Let $|S|=s$ and assume that $s>2^{200k}$, as otherwise we are done. Using Theorem \ref{thm:long paths}, we can find a copy of $P_t^{20k}$ with $t= s/2^{160k}>41k$. Look at the set of the first $20k$ vertices in this path; since every vertex has at least $|A|/20$ in-neighbors in $A$, we conclude, by Lemma \ref{prop:KST}, that there exists a set of $k$ vertices among the first $20 k$ vertices in the found copy of $P_t^{20k}$, such that their common in-neighborhood in $A$ is at least of size $\left(\frac{1}{20e}\right)^k |A|>\frac{1}{2^{6k}}|A|$. 
Since this $k$-set is also transitive by construction, it forms an $A$-tail. Now we remove the other $19k$ vertices from the beginning of the path. Similarly, among the last $20k$ vertices we find a $B$-head of size $k$ and we remove the other $19k$ vertices.  Note that the obtained structure $P$ is a $k$-th power of a path starting with an $A$-tail and ending with a $B$-head, or in other words, it is an $(A,B,k)$-chain of length $t-38k$.

We repeat this procedure, now only using vertices in $S\setminus V(P)$. Let $s'=|S\setminus V(P)|$, so we get a new $(A,B,k)$-chain of length $\frac{s'}{
2^{160k}}-38k>\frac{s'}{
2^{161k}}$, if $s'>2^{200k}$. We continue doing this until at most $2^{200k}$ vertices are left, which enables us to find $20k$-th powers of paths of length at least $\frac{2^{200k}}{2^{160k}}>41k$ in every step, and turn them into $(A,B,k)$-chains. 
Notice that after finding each $(A,B,k)$-chain, the number of uncovered vertices is at most $1-\frac{1}{2^{161k}}$ times the number of uncovered vertices in the previous step; this means that in the end we find at most 
\[
-\log_{1-\frac{1}{2^{161k}}}s=-\frac{\log s}{\log (1-\frac{1}{2^{161k}})}<2^{200k}\log s
\]
disjoint $(A,B,k)$-chains when we stop, which finishes the proof.
\end{proof}
The next result is one of the main tools in our proofs; it is used to link vertices from the end of one chain to the beginning of another chain.
\begin{lem} [The linking lemma]\label{lem:linking}
Let $\delta > 0$, $k\geq 2$ and $n>\left(\frac{2}{\delta}\right)^{1000k}$.  Let $T$ be an $n$-vertex $\delta$-cut-dense tournament with $\delta^0(T)>(\frac{1}{4}-\frac{\delta}{100})n$.
Suppose $M,N\subseteq V(T)$ are of size $k$, and $M$ has at least $\frac{n}{2^{20k}}$ common out-neighbors, while $N$ has at least $\frac{n}{2^{20k}}$ common in-neighbors. Let $\pi$ be an ordering of $V(T)\setminus (M\cup N)$. Then, there is a $k$-th power of a path $P$ starting in $M$ and ending in $N$, so that the vertices of $P\setminus(M\cup N)$ can be partitioned into a collection of sets $\mathcal{S}=\{S_1,\ldots, S_m\}$ which is $10k$-apart in $\pi,$ while $|S_i| = k$ and $m\leq \log n$. 
\end{lem} 
\begin{proof}
We let $c=\delta/100$ and we start with defining the set $M_1$ to be the set of common out-neighbors of $M$ and the set $N_1$ to be the set of common in-neighbors of $N$; hence $|M_1|, |N_1| \geq \frac{n}{2^{20k}}$. Next, we define the sets $M_i$ and $N_i$ for $i\geq 2$ as follows. We first let $M^i=\bigcup_{j\leq i} M_j$ and $N^i=\bigcup_{j\leq i}N_j$ for each $i\geq1$, and let:
\[
M_i = \{x \notin M \cup N\cup M^{i-1}\mid d^-_{M^{i-1}}(x)\geq c|M^{i-1}|\}\enspace\text{and}\enspace N_i = \{x \notin M\cup N \cup N^{i-1} \mid d^{+}_{N^{i-1}}(x) \geq c|N^{i-1}|\},
\]
i.e. $M_i$ is the set of vertices outside of $M\cup N\cup M^{i-1}$ with many in-neighbors in $M^{i-1}$, and $N_i$ is the set of vertices outside of $M\cup N\cup N^{i-1}$ with many out-neighbors in $N^{i-1}$. Furthermore, by bounding $\sum_{x \in M^{i-1}} d^{+}(x)$ and using that vertices outside $M\cup N\cup M^i$ have in-degree at most $c|M^{i-1}|$ in $M^{i-1}$ we get:
$$\left( \frac{1}{4} - c\right)n|M^{i-1}| \leq \sum_{x \in M^{i-1}} d^{+}(x) \leq c|M^{i-1}|n + |M^{i-1}|(|M|+|N|+|M_{i}|) + \frac{|M^{i-1}|^2}{2},$$
which gives: 
\begin{equation}\label{eq:doubleCount}
|M_{i}| \geq n \left(\frac{1}{4} - 2c \right) - 2k - \frac{|M^{i-1}|}{2}
\end{equation}
and by adding $|M^{i-1}|$ to both sides: $|M^{i}| \geq n \left(\frac{1}{4} - 2c \right) - 2k + \frac{|M^{i-1}|}{2}$. From this recursive relation we get that:
$$|M^{i}| \geq n \left(1 - \frac{1}{2^{i-1}} \right) \left(\frac{1}{2} - 4c \right) + \frac{n/2^{20k}}{2^{i-1}} - 4k.$$
Define $i_M$ to be minimal such that $|M^{i_M}| \geq n \left(\frac{1}{2} - \frac{\delta}{8} \right)$ and note that $i_M \leq \log_2 \left(\frac{100}{\delta} \right)$.
Also, note that the recursive relation (\ref{eq:doubleCount}) and the fact that $|M^{i-1}|\leq n \left(\frac{1}{2} - \frac{\delta}{8} \right)$ for $i \leq i_M$ gives for $2 \leq i \leq i_M$
\begin{equation}\label{eq:Bound on M_i}
\left|M_{i} \right| \geq n \left(\frac{\delta}{16} - 2c\right) - 2k \geq \frac{\delta n}{32}.
\end{equation}
Note also that every vertex $x \in M_{i}$ has $d_ {M_{i-1}}^{-}(x)  \geq c|M_{i-1}|$; this is because $x \notin M_{i-1}$ and so, we have both $d^{-}_ {M^{i-1}}(x)  \geq c|M^{i-1}|$ and $d_{M^{i-2}}^{-}(x) < c|M^{i-2}|$ which implies the desired inequality since $M_{i-1} = M^{i-1} - M^{i-2}$. We can define $i_N$ analogously to $i_M$, and all bounds which hold for $M_i,M^i,i_M$ hold also for $N_i$, $N^i$ and $i_N$.

\noindent \textbf{Case 1: there exist $i_1 \leq i_M$ and $i_2 \leq i_N$ such that $|M_{i_1} \cap N_{i_2}| > n^{1/4}$}.

Let $i_1+i_2$ be minimal so that the above inequality holds. Let $S = M_{i_1}\cap N_{i_2}$ and let $t=i_1i_2 \cdot n^{1/4} < n^{0.3}$, so that $|M^{i_1-1}\cap N^{i_2}|, |M^{i_1}\cap N^{i_2-1}| \leq t$. Now, for all $i < i_1$ define the set $K_i:=M_i\setminus N^{i_2}$ and for all $i < i_2$, define $L_i:=N_i\setminus M^{i_1}$. By the inequality (\ref{eq:Bound on M_i}), we have that $|L_i|,|K_i|>\delta n/32-t>\delta n/64$ for $i \geq 2$ and 
$|L_1|,|K_1|> \frac{n}{2^{20k}}-t>\frac{n}{2^{21k}}$. Furthermore, note that from the discussion after (\ref{eq:Bound on M_i}), we also have that for every $x\in S$,
\begin{equation}\label{eq:density1}
d_{K_{i_1-1} }^{-}(x)  \geq c \left|M_{i_1-1} \right| - t \geq \frac{c}{2} \left|K_{i_1-1} \right|
\end{equation}
and similarly for every $i<i_1$ and $x\in K_i$ it holds $d_{K_{i-1}}^-(x)\geq \frac{c}{2}|K_{i-1}| $. The analogous bounds hold for each $L_i$.

Now we will find the $k$-th power of a path which connects $M$ and $N$, and which satisfies the conditions of the lemma; in order to do this we find sets $Z,X_1,\ldots,X_{i_1-1}, Y_1,\ldots, Y_{i_2-1}$  of size $k$, such that all of them induce transitive tournaments, with $Z\subset S$, $X_i\subset K_{i_1-i}$ and $Y_i\subset L_{i_2-i}$, and where 
$X_{i_1-1},\ldots X_1,Z,Y_1,\ldots Y_{i_2-1}$ induces a $k$-blowup of a path, where the listed sets are the blobs of the blowup in this order. Furthermore, the collection of blobs is $10k$-apart in $\pi$. Since $X_{i_1-1}$ is in the common out-neighborhood of $M$, and $Y_{i_2-1}$ is in the common in-neighborhood of $N$, this will give the $k$-th power of the statement, where $\mathcal{S}=\{X_{i_1-1},\ldots X_1,Z,Y_1,\ldots Y_{i_2-1}\}$.

Set $F:=\emptyset$ in the beginning and as we find a new blob we add some \emph{forbidden} vertices to $F$, which we are not allowed to use in subsequent blobs. Every time we find a new blob we will add at most $30k^2$ vertices to $F$, so at any point $F$ is of size at most 
$(i_1+i_2)30k^2<60k^2 \log_2(100/\delta)<k\log_2 n$ (using that $n>(2/\delta)^{1000k}$).
We begin by finding the set $Z\subset S$.
Look at the bipartite graph $H$ on $(S, K_{i_1-1}\times L_{i_2-1})$ where there is an edge between $s$ and $(x,y)$ if and only if $xs$ and $sy$ are edges in $T$. Every vertex in $S$ has at least $\frac{c^2}{4}|K_{i_1-1}\times L_{i_2-1}|$ neighbors by (\ref{eq:density1}) and thus, $H$ has at least $\frac{c^2}{4}|S||K_{i_1-1}\times L_{i_2-1}|$ edges. Note the following direct application of Lemma \ref{DRC for sets}.
\begin{claim}\label{Case 1}
Let $G$ be a bipartite graph with parts $A,B$ of sizes larger than $n^{1/5}$ and such that $e(A,B) \geq \frac{c^2}{4} |A||B|$. Then, there exists a set $U\subseteq A$ of size at least $2^{2k}$ such that at most a fraction of $1/2^{2k^2}$ of its $k$-sets have less than $|B|^{2/3}$ common neighbors in $B$.
\end{claim}
\begin{proof}
Take the parameters of Lemma \ref{DRC for sets} to be $a=|A| \geq n^{1/5}$, $b=|B| \geq n^{1/5}$, $\gamma=\frac{c^2}{4}$, $m=b^{2/3}$,$s=2^{2k}$. Indeed, it holds that
\[\binom{a\gamma^k}{k}\geq \frac{a^k\gamma^{k^2}}{k^k}\geq 2s^k\frac{a^ke^k }{k^k}\cdot \left(\frac{c^{2k}}{2\cdot 4^kse}\right)^k\geq 2s^k\binom{a}{k}b^{-k/3} = 2s^k\binom{a}{k}\left(\frac{m}{b}\right)^k
\]
where the last inequality holds since $b\geq n^{1/5} \geq 
\frac{(10 s 4^k)^3}{c^{6k}}$, where we used that $n>\left(\frac{2}{\delta}\right)^{1000k}$ and $c=\delta/100$. Further, we are done by noting that
\[
\binom{a\gamma^k}{k}\geq\left( \frac{a\gamma^{k}}{k}\right)^k\geq n^{k/10}\geq 2s^k .
\]
\end{proof}
Applied to the graph $H$, this claim implies the existence of a subset $U \subseteq S$ of size at least $2^{2k}$ such that at least a $(1-1/2^{2k^2})$-fraction of its $k$-sets have at least $|K_{i_1-1}\times L_{i_2-1}|^{2/3} / |L_{i_2-1}|>n^{1/4}$ in-neighbors in $K_{i_1-1}$ and at least $n^{1/4}$ out-neighbors in $L_{i_2-1}$. Since $U$ is of size at least $2^{2k}$ and the fraction of its $k$-sets which induce transitive tournaments is, by Lemma \ref{lem:transitive subtournament}, larger than $1/2^{k^2}$, this means that there is a transitive $k$-set $Z$ in $S$ with at least $n^{1/4}$ common in-neighbors in $K_{i_1-1}$ and at least $n^{1/4}$ common out-neighbors in $L_{i_2-1}$.
We now update $F$ by setting $F:=I_{10k}(Z,\pi)$, i.e. $F$ is the set of vertices which are at distance at most $10k$ in $\pi$ from at least one vertex in $Z$. Hence, $F$ is of size at most $|F|\leq k(20k+1)\leq 30k^2$.
 
We continue by now finding the set $X_1\subseteq K_{i_1-1}\setminus F$ among the at least $n^{1/4}-|F|>n^{1/5}$ vertices which are common in-neighbors of $Z$. Crucially, notice that every vertex in 
$K_{i_1-1}$ has at least $\frac{c}{2}|K_{i_1-2}|-|F|>\frac{c}{4}|K_{i_1-2}|$ in-neighbors in $K_{i_1-2}\setminus F$. Therefore, the bipartite graph induced by $\overrightarrow{E}\big[K_{i_1-2}\setminus F,N_{K_{i_1-1}}^-(Z)\setminus F\big]$ has density at least $\frac{c}{4}$ and parts of sizes larger than $n^{1/5}$. Thus, we can apply Claim \ref{Case 1} again to find a transitive $k$-set $X_1 \subseteq N_{K_{i_1-1}}^-(Z)\setminus F$ which has at least $|K_{i_1-2}\setminus F|^{2/3} > n^{1/4}$ common in-neighbors in $K_{i_1-2}\setminus F$. 
We update $F$ accordingly, i.e. $F:=F\cup I_{10k}(X_1,\pi)$, and so, $F$ increases by at most $30k^2$.  
 
We repeat this process to find $X_2,\ldots,X_{i_1-1}$. Suppose we have found the sets $X_1,\ldots,X_{i-1}$ and we know that $X_{i-1}$ has more than $n^{1/5}$ common in-neighbors in $K_{i_1-i}\setminus F$; in order to find a $k$-set $X_i$ among those neighbors, such that $X_i$ has at least $n^{1/4}$ common in-neighbors in $K_{i_1-i-1}\setminus F$, we consider the bipartite graph induced by $E\big[K_{i_1-i-1}\setminus F, N_{K_{i_1-i}}^-(X_{i-1})\setminus F\big]$ and apply Claim \ref{Case 1}. Indeed, at any point we will have that $|F|<(i_1+i_2)30k^2<k \log_2 n$. So, by the bound below inequality (\ref{eq:density1}), every vertex in $K_{i_1-i}$ has at least $\frac{c}{2}|K_{i_1-i-1}|-|F|>\frac{c}{4}|K_{i_1-i-1}|$ in-neighbors in $K_{i_1-i-1}\setminus F$ and thus, the bipartite graph will have density at least $\frac{c}{4}$. To finish the step, we update $F$ accordingly, by adding to it the vertices in $I_{10k}(X_i,\pi)$. The same process is done to find the sets $Y_i$ for $1\leq i< i_2$, although we now start with the set $F$ of forbidden vertices we already have. 

This finishes the first case analysis. Now, suppose that $|M_{i} \cap N_{j}| \leq n^{1/4}$ for all $i\leq i_M$ and $j\leq i_N$; this implies that $|M^{i_M} \cap N^{i_N}|\leq i_Mi_N \cdot n^{1/4}<n^{0.3}$.

\noindent \textbf{Case 2: $|M^{i_M} \cap N^{i_N}| < n^{0.3}$}

Similarly to before, we define the sets
$K_i=M_i\setminus N^{i_N}$ for all $i\leq i_M$ and $L_i=N_i\setminus M^{i_M}$ for all $i\leq i_N$. Additionally, we define the sets $K^i=\bigcup_{j\leq i}K_j$ and $L^i=\bigcup_{j\leq i}L_i$. Note that by (\ref{eq:Bound on M_i}) we have $|K_i|>|M_i|-n^{0.3}>\delta n/64$ and similarly, 
$|L_i|>\delta n/64$. Furthermore, we have the bounds $|K^{i_M}|,|L^{i_N}|> n \left(\frac{1}{2} - \frac{\delta}{8} \right) - n^{0.3}> n \left(\frac{1}{2} - \frac{\delta}{7} \right)$.
Using the fact that $T$ is $\delta$-cut-dense from the statement of the lemma, we get that 
$$
\overrightarrow{e}(K^{i_M},L^{i_N})\geq \frac{\delta n^2}{4}- 2\frac{\delta n}{7} \cdot \frac{n}{2} \geq \frac{\delta n^2}{10} \geq \frac{\delta}{3} |K^{i_M}||L^{i_N}|
$$
and in particular, there exists a pair $\left( K_{i_1}, L_{i_2}\right)$ such that the density of edges going from $K_{i_1}$ to $L_{i_2}$ is at least
$\delta/3>c$, i.e. 
\begin{equation}\label{density KL}
    \overrightarrow{e}(K_{i_1},L_{i_2})\geq c |K_{i_1}||L_{i_2}|.
\end{equation}
 Also note that $d^-_{K_i}(x)>c|M_i|-n^{0.3}>\frac{c}{2}|K_i|$ for every $x\in K_{i+1}$ and $i<i_1$, and likewise we have $d^+_{L_i}(x)>\frac{c}{2}|L_i|$ for $i<i_2$ and $x\in L_{i+1}$. 

Similarly to before, we find sets $X_{i_1-1},\ldots,X_1,Z,Y_1,\ldots,Y_{i_2}$  which form a $k$-blowup of a path in this order, with the same properties as before, but now with $X_i\subseteq K_{i_1-i}$, $Z \subseteq K_{i_1}$ and $Y_i\subseteq L_{i_2-i+1}$ for all $i$.
We first set our set of forbidden vertices $F$ to be empty again.

We will only show how to find the set $Z$, as when we have it, the linking proceeds in the same way as before. So, take the bipartite graph $H$ with parts $K_{i_1}$ and $K_{i_1-1}\times L_{i_2}$, which have both size larger than $n^{1/5}$, where there is an edge between $x$ and $(u,v)$ if $ux$ and $xv$ are edges in $T$. Notice that the number of edges in this bipartite graph is $$\sum_{x\in K_{i_1}} |N^-(x)\cap K_{i_1-1}|\times  |N^+(x)\cap L_{i_2}|\geq \frac{c}{2}|K_{i_1-1}|\sum_{x\in K_{i_1}}|N^+(x)\cap L_{i_2}|>\frac{c^2}{2}|K_{i_1}||K_{i_1-1}||L_{i_2}|$$ 
where the last inequality follows from (\ref{density  KL}) together with 
$\sum_{x\in K_{i_1}}|N^+(x)\cap L_{i_2}|=\overrightarrow{e}(K_{i_1},L_{i_2})$.
We can then apply Claim \ref{Case 1} to $H$, and as before, get the existence of a transitive $k$-set $Z \subseteq K_{i_1}$ which has at least $n^{1/4}$ common in-neighbors in $K_{i_i-1}$ and at least $n^{1/4}$ common out-neighbors in $L_{i_2}$. We update $F$ by setting $F:=F\cup I_{10k}(Z,\pi)$.
We continue to find the remaining sets $X_i,Y_i$ in the same way as in the first case, giving us the desired $k$-th power of a path linking $M$ to $N$.
\end{proof}

\section{Cut-dense tournaments}\label{cutdense}
In this section we show that if a tournament is cut-dense, then a minimum semidegree slightly below $n/4$ is already enough to guarantee the existence of the $k$-th power of a Hamilton cycle.
\Cutresult*
\begin{proof}

Firstly, note that we can use Lemma \ref{lem:partition lemma} (with $A=B=V(T)$) to partition the vertices of $T$ into  $t \leq 2^{600k}\log n$ disjoint $3k$-chains $C_1,C_2,\ldots, C_t$ and  a set $R$ of size at most $2^{600k}$. In the next subsection, we cover $R$ with disjoint $k$-chains of length $2k+1$. Some of the vertices in these $k$-chains might be contained in one of $C_1,C_2,\ldots,C_t$; in this case, we remove each such vertex from the corresponding $C_i$. In the end, we will denote by $C_i'$ what is left of $C_i$, for all $i\geq 1$. The covering of $R$ is done so that each $C_i'$ has the same first and last $k$ vertices as $C_i$.

\subsection{Covering $R$ with $k$-chains}
Let $F$ be the set of vertices in $T$ which are either in $R$ or in the first or last $k$ vertices in one of the $3k$-chains $C_1,C_2,\ldots C_t$, implying that $|F|<2^{600k}+2k\cdot 2^{600k}\log n$. Let $\pi$ be the ordering on $V=V(T)\setminus F$ such that all vertices in each $C_i$ come before all vertices in $C_{i+1}$, and inside of each $C_i$ the vertices inherit the ordering of the chain. For every vertex $v\in R$, let $A_v=N^-(v)\cap V$ and let $B_v=N^+(v)\cap V$.
We have that $|A_v|+|B_v|=n-|F|$, and $|A_v|,|B_v|\geq \delta^0(T)-|F|>\frac{n}{5}.$

We repeat the following procedure until all vertices in $R$ are covered with disjoint $k$-chains of length $2k+1$. Let $w_1\in R$ be an uncovered vertex. By Lemma \ref{lem:anycutisdense} applied to $T$, we get that 
\begin{equation}\label{eq:A_w}
    \overrightarrow{e}(A_{w_1},B_{w_1})>\frac{\delta}{16}|A_{w_1}||B_{w_1}|-|F|n>\frac{\delta}{32}|A_{w_1}||B_{w_1}|
\end{equation}

By using Lemma \ref{DRC for sets}, we obtain the following claim. 
\begin{claim}\label{cl:absorption}
There exist a set $U\subseteq A_{w_1}$ of size $|U|\geq 2^{30k}$ such that at most a $\frac{1}{ 2^{30k^2}}$-fraction of its $k$-sets has less than $2^{30k}$ out-neighbors in $B_{w_1}$.
\end{claim}
\begin{proof}
Apply Lemma \ref{DRC for sets} with parameters $a=|A_{w_1}|\geq \frac{n}{5}$, $b=|B_{w_1}|\geq \frac{n}{5}$, $\gamma=\frac{\delta}{32}$, $m=2^{30k}$,$s=2^{30k}$ to the graph with parts $(A_{w_1},B_{w_1})$ . Indeed, it holds that
\[\binom{a\gamma^k}{k}\geq \frac{a^k\gamma^{k^2}}{k^k}\geq
2s^k\frac{a^ke^k }{k^k}\cdot \left(\frac{\delta^{k}}{2\cdot 32^kse}\right)^k\geq 2s^k\binom{a}{k}\left(\frac{n^{-\frac{1}{1000}}}{2^{40k}}\right)^k\geq 2s^k\binom{a}{k}\left(\frac{m}{b}\right)^k
\]
where the last inequality holds since $b\geq\frac{n}{5}$ so $\frac{m}{b}<\frac{1}{\sqrt{n}}$, and it holds that $n>2^{1000k}$. We are done since also
\[
\binom{a\gamma^k}{k}\geq\left( \frac{a\gamma^{k}}{k}\right)^k\geq (\sqrt{n})^k\geq 2s^k.
\]
\end{proof}

By Lemma \ref{inclusion-exlusion} (with $A=V(T)$) it holds that more than a $1/2^{30k^2}$ fraction of the $k$-sets in $U$ are tails, hence there is a $k$-set $X_1\subseteq U$ which is a tail and has at least $2^{30k}$ out-neighbors in $B_{w_1}$.
Finally, again by Lemma \ref{inclusion-exlusion}, in the set of the at least $2^{30k}$ common out-neighbors of $X_1$ in $B_{w_1}$, there is a head $Y_1$ of size $k$. This gives that  $X_1\cup\{w_1\}\cup Y_1$ induces a $k$-chain, which we call $C_{t+1}'$.

We update the sets $F,V$ and $A_v,B_v$ for each $v$, by setting $F:=F\cup I_{3k}(C_{t+1}',\pi)$, $V:=V\setminus F$, $A_v:=N^-(v)\cap V$ and $B_v:=N^+(v)\cap V$.
We continue with an uncovered vertex $w_2$, we find the $k$-chain $C_{t+2}'=X_2\cup\{w_2\}\cup Y_2$, and again update $F:=F\cup I_{3k}(C_{t+2}',\pi)$ as well as $V:=V\setminus F$. Furthermore, we update the sets $A_v:=N^-(v)\cap V$ and $B_v:=N^+(v)\cap V$.

We repeat this procedure until we covered all vertices in $R$, noting that at every step it holds that $|A_v|+|B_v|\geq n-|F|$ and $|A_v||B_v|>\frac{n}{5}$, so equation (\ref{eq:A_w}) holds at every step. Indeed, we always have $$|F|\leq 2k\cdot2^{600k}\log n+ (20k^2+1)|R|<\sqrt{n},$$ since $|I_{3k}(C_{t+i},\pi)|\leq 
(2k+1)(6k+1)<20k^2$ as $|C_{t+i}'|=|X_i\cup\{w_i\}\cup Y_i|=2k+1$, so that in each of the $|R|$ steps, we add at most $20k^2$ new vertices to $F$. Note that by construction the collection of sets $\{X_i\cup Y_i\mid i\in[|R|]\}$ obtained in the covering process is $3k$-apart in $\pi$ (hence also in each $C_i$). We remove the vertices in these sets from the $3k$-chains $C_1,C_2,\ldots C_t$, and are left with $C_1',C_2',\ldots C_t'$. In what follows, we use the obtained sets $F$ and $V$.
\subsection{Linking the chains}\label{cutdense2}
To finish the proof we will link the structures from the previous subsection. We will link the last $k$ vertices of each $C_i'$ to the first $k$ of $C_{i+1}'$, and the last $k$ of $C_{t+|R|}'$ to the first $k$ of $C_1'$. 

First we want to link $C_1'$ to $C_2'$; for this we apply Lemma \ref{lem:linking} to the tournament $T[V\cup H_1\cup H_2]$ with the ordering $\pi$, where $H_1$ and $H_2$ are respectively the last $k$ vertices in $C_1'$ and first $k$ of $C_2'$. Thus we get a $k$-th power of a path $P$ between $H_1$ and $H_2$, so that $P\setminus (H_1\cup H_2)$ can be partitioned into sets $S_1,\ldots,S_m$ which are $3k$-apart in $\pi$. Now we update $F$ by setting $F:=F\cup I_{3k}(S_1\cup\ldots\cup S_m,\pi)$, and also $V:=V\setminus F$. Notice that this means that $V$ decreased by at most $10k^2\log n$, since $|S_1\cup\ldots\cup S_m|\leq k\log n$, so that $|I_{3k}(S_1\cup\ldots\cup S_m,\pi)|\leq (6k+1)k\log n\leq 10k^2\log n$.
We repeat this process until we made all connections; suppose we are about to connect $C_i'$ to $C_{i+1}'$, so we have deleted at most 
$(t+|R|)10k^2\log n \leq 20k^22^{600k}\log^2 n \leq \sqrt{n}$ vertices from $V$ during the linking process, meaning that 
altogether we deleted $|F| \leq 2\sqrt{n}$ vertices and therefore 
$$\delta^0(T[V\cup H_{2i-1}\cup H_{2i}])>n/4-\frac{\delta}{200}n-2\sqrt{n}>\left(\frac{1}{4}-\frac{\delta}{150}\right)n$$
so we can continue applying Lemma \ref{lem:linking}. Indeed, let $T'=T[V\cup H_{2i-1}\cup H_{2i}]$ and let $n'=|T'|$. Since we have that $|V|\geq n -2\sqrt{n}$, it is easy to see that $T'$ is $\delta'$-cut-dense with $\delta'=0.99\delta$ and also $n'>\left(\frac{2}{\delta'}\right)^{1000k}$. Furthermore $\delta^0(T')>(\frac{1}{4}-\frac{\delta'}{100})n'$, and $N^+_{T'}(H_{2i-1})>\frac{n'}{2^{20k}}$, while  $N^-_{T'}(H_{2i})>\frac{n'}{2^{20k}}$.

Notice also that the collection of all sets $S_i$ used for creating the linking paths, together with the collection of all $X_i\cup Y_i$ used for covering $R$, form a collection that is $3k$-apart in $\pi$ by construction. Note also that each set in this collection is of size at most $2k$; this means that when we remove the vertices used in the linking paths from the structures $C_1',\ldots, C_{t+|R|}'$, we are left with $k$-th powers of paths by Lemma \ref{lem:destroying powers}, and we call them $C_1'',\ldots, C_{t+|R|}''$. Note that by construction, the last $k$ vertices of each $C_i''$ are linked to the first $k$ of $C''_{i+1}$ (and similarly, $C''_{t+|R|}$ to $C''_1$) by disjoint $k$-th powers of paths (i.e. the linking paths). Furthermore, the collection $C_1'',\ldots, C_{t+|R|}''$ together with the constructed links cover the whole vertex set of $V$. This gives the $k$-th power of a Hamilton cycle, which completes the proof.
\end{proof}

\begin{remar}\label{rem:hamiltonian path}
Note that an immediate result which follows from our proof is that in any tournament with the same conditions as in Theorem \ref{thm:cut result}, we can find a spanning $k$-chain (which we get before the last step of the linking process). We will use this observation in consequent sections.
\end{remar}

\section{A bridge}\label{abridge}
In this section, we will prove a result which, as discussed in Section \ref{sec:proof overview}, will allow us to construct the so-called bridges needed in the proof of Theorem \ref{maintheorem}.
\begin{thm}\label{bridgeslemma}
There exists a constant $C = C_k$ such that the following holds. Let $T$ be a tournament on $n$ vertices and $X,Y$ a partition of its vertex set such that $\overrightarrow{e}(X,Y)\geq Cn^{2- 1/ \lceil k/2 \rceil}$. Then there exists a sequence $(x_1, \ldots, x_k, y_1, \ldots, y_k) \in X^k \times Y^k$ which is the $k$-th power of a path.
\end{thm}

First, we need a few definitions. Let $T$ be a tournament which has a partition $(X,Y)$ such that $\overrightarrow{e}(X,Y) = \delta |X||Y|$ for some $\delta$. For every vertex $x \in X$, we define $Y(x)$ to be the set of vertices $v \in N^{+}_Y(x)$ such that $|N^{-}_X(v) \cap N^{-}_X(x)| \geq \frac{\delta|X|}{8}$. Analogously, for $y \in Y$, define $X(y)$ to be the set of vertices $u \in N^{-}_X(y)$ such that $|N^{+}_Y(u) \cap N^{+}_Y(y)| \geq \frac{\delta|Y|}{8}$.
\begin{lem}\label{countinglemma}
Suppose $\delta|X|, \delta|Y| > 2^{10} $. Then, there exist vertices $x \in X, y \in Y$ such that
$$\frac{|X(y)|}{|X|}, \frac{|Y(x)|}{|Y|} \geq \frac{\delta}{8}$$
\end{lem}
\begin{proof}
By symmetry, we only prove the existence of $x$. First, note that we can, without loss of generality, re-orient some edges in $\overrightarrow{E}(X,Y)$ and have that every vertex in $Y$ has either no in-neighbors in $X$ or at least $\frac{\delta |X|}{2}$ of them, and that $\overrightarrow{e}(X,Y) \geq \frac{\delta}{2}|X||Y|$. Now, let $x \in X$ be chosen uniformly at random - we will show that $\mathbb{E}[Y(x)] \geq \frac{\delta |Y|}{8}$, which proves the lemma. 

Now, let $v \in Y$ and consider the tournament induced by $N^{-}_X(v)$, which is either empty or of size at least $\frac{\delta|X|}{2} \geq 50$. It is easy to see that at least a quarter of its vertices have at least a quarter of the vertex set in their in-neighborhood. Therefore, a quarter of the vertices $u \in N^{-}_X(v)$ are such that $|N^{-}_X(v) \cap N^{-}_X(u)| \geq \frac{\delta|X|}{8}$, i.e., $v \in Y(u)$. Hence, $\mathbb{P}(v \in Y(x)) \geq \frac{ |N^{-}_X(v)|}{4|X|}$ and thus,
$$\mathbb{E}[Y(x)] = \sum_{v \in Y} \mathbb{P}(v \in Y(x)) \geq \frac{1}{4|X|} \sum_{v \in Y} |N^{-}_X(v)| = \frac{\overrightarrow{e}(X,Y)}{4|X|} \geq \frac{\delta |Y|}{8}$$
\end{proof}

We can now prove the following claim, from which Theorem \ref{bridgeslemma} will be easy to derive. The proof uses induction together with the following observation. Suppose that $(a,b,c,d) \in X^2 \times Y^2$ is the square of a path and $X' \subseteq X,Y' \subseteq Y$ are sets such that $a \rightarrow X' \rightarrow \{b,c\}$ and $\{b,c\} \rightarrow Y' \rightarrow d$. Then, if $(x_1, \ldots, x_{k-2}, y_1, \ldots, y_{k-2}) \in (X')^{k-2} \times (Y')^{k-2}$ is the $(k-2)$-th power of a path, we have that $(a,x_1, \ldots, x_{k-2},b,c, y_1, \ldots, y_{k-2},d)$ is the $k$-th power of a path.

\begin{claim}\label{generalbridgesclaim}
Let $k$ be even and $\delta^{\frac{k}{2}}|X|, \delta^{\frac{k}{2}}|Y| > 2^{10k^2}$. Then $X^k \times Y^k$ contains a $k$-th power of a path.
\end{claim}
\begin{proof}
We prove this by induction on $\frac{k}{2}$. Take the base case $k = 2$ and suppose that $\delta|X|, \delta|Y| > 2^{10}$. By Lemma \ref{countinglemma}, there exists a vertex $b \in X$ such that $Y(b) \geq \frac{\delta |Y|}{8} \geq 2$. Therefore, there exists an edge $cd$ with both $c,d \in Y(b)$. By definition of $Y(b)$, we have $| N^{-}_X(b) \cap N^{-}_X(c)| \geq \frac{\delta |X|}{8} \geq 1$ and so, there is a vertex $a \in N^{-}_X(b) \cap N^{-}_X(c)$. Notice that the sequence $(a,b,c,d)$ is the square of a path. 

Now, let $k \geq 4$ and suppose the statement is true for all smaller even $k$. Assume that $\delta^{\frac{k}{2}}|X|, \delta^{\frac{k}{2}}|Y|> 2^{10k^2}$. By the observation we made earlier, we need only to find a sequence $(a,b,c,d) \in X^2 \times Y^2$ which is the square of a path and sets $X' \subseteq X,Y' \subseteq Y$ such that $a \rightarrow X' \rightarrow \{b,c\}$, $\{b,c\} \rightarrow Y' \rightarrow d$ and 
\begin{equation}\label{indstep}
\left(\delta' \right)^{\frac{k}{2}-1} |X'| , \left(\delta' \right)^{\frac{k}{2}-1} |Y'| > 2^{10(k-2)^2}
\end{equation}
where $\delta'$ is such that $\overrightarrow{e}(X',Y') = \delta'|X'||Y'|$.

To that end, we first apply Lemma \ref{countinglemma} again to $T$. It implies the existence of a vertex $b \in X$ such that $Y(b) \geq \frac{\delta |Y|}{8}$. Now, take a vertex $d$ in $Y(b)$ which has at least a quarter of the vertices in $Y(b)$ as in-neighbors. Define $Y_1 = Y(b) \cap N^{-}(d)$ so that 
$$|Y_1| \geq \frac{1}{4}|Y(b)| \geq \frac{\delta}{32}|Y|$$ 
By definition, every vertex in $Y_1$ has at least $\frac{\delta }{8}|X|$ in-neighbors in $X_1 := N^{-}_X(b)$. Thus, we have $\overrightarrow{e}(X_1,Y_1) \geq \frac{\delta}{8}|Y_1||X|$, which, defining $\delta_1$ to be such that $\overrightarrow{e}(X_1,Y_1) = \delta_1|X_1||Y_1|$, gives
$$\delta_1 \geq \frac{\delta|X|}{8|X_1|} \geq \frac{\delta}{8}.$$

We now consider the tournament induced by $X_1 \cup Y_1$ and apply Lemma \ref{countinglemma} to it. Note that since $k \geq 4$ the conditions are still satisfied since $\delta_1 |Y_1| \geq \frac{\delta}{8}|Y_1| \geq \frac{\delta^2}{2^8}|Y| > 2^{10}$ and $\delta_1 |X_1| \geq \frac{\delta}{8}|X| > 2^{10}$. Therefore, there exists a vertex $c \in Y_1$ such that $X_1(c) \geq \frac{\delta_1 }{8} |X_1|$. Much like before, take a vertex $a$ in $X_1(c)$ which has at least a quarter of its vertices as out-neighbors and so, if we define $X_2 = X_1(c) \cap N^{+}(a)$ we have $$|X_2| \geq \frac{1}{4}|X_1(c)| \geq  \frac{\delta_1 }{32} |X_1| \geq \frac{\delta}{2^8} |X|$$
Letting $Y_2 = N^{+}_{Y_1}(c)$, we have, by definition, that $\overrightarrow{e}(X_2,Y_2) \geq \frac{\delta_1}{8} |X_2||Y_1|$. This in turn gives that
$$\delta_2 \geq \frac{\delta_1|Y_1|}{8|Y_2|} \geq \frac{\delta_1}{8} \geq \frac{\delta}{2^6}$$
where $\delta_2$ is defined so that $\overrightarrow{e}(X_2,Y_2) = \delta_2|X_2||Y_2|$. In particular, $\delta_2|Y_2| \geq \delta_1|Y_1|/8$. We now take $X' = X_2$ and $Y' = Y_2$ and $\delta' = \delta_2$. In order to verify the condition (\ref{indstep}), note that by the previously displayed inequalities, we have $\delta' \geq \frac{\delta}{2^6}$, 
$\delta'|X'|=\delta_2|X_2| \geq \frac{\delta^2}{2^{14}} |X|$  and $\delta' |Y'|=\delta_2|Y_2| \geq \frac{\delta_1}{8}|Y_1| \geq \frac{\delta^2}{2^{11}}|Y|$. These bounds can be used to note
$$\left(\delta' \right)^{\frac{k}{2}-1} |X'|  = \left(\delta' \right)^{\frac{k}{2}-2} \cdot \delta' |X'| \geq \left(\frac{\delta}{2^6} \right)^{\frac{k}{2}-2} \cdot \frac{\delta^2}{2^{14}} |X| \geq \frac{\delta^{\frac{k}{2}}}{2^{10k}}|X| > 2^{10(k-2)^2}$$
and the same occurs for $Y'$.
\end{proof}

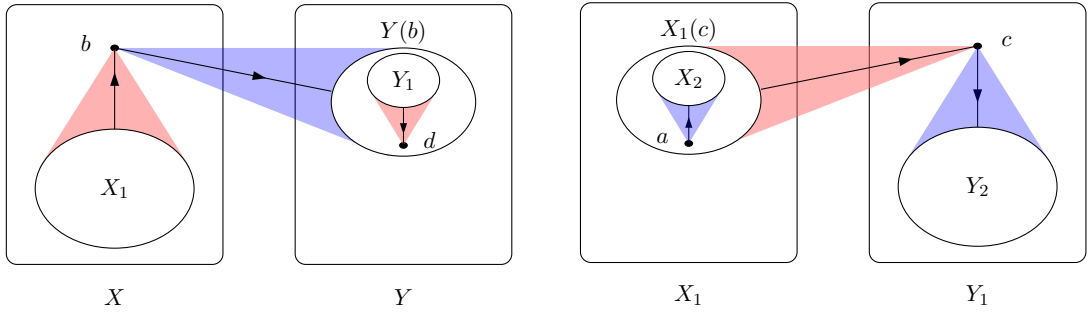
\begin{figure}[h]
\begin{tikzpicture}[scale = 1.2, xscale=1.6, yscale=1.2]
\draw[rounded corners] (-3.75,-0.2) rectangle (-2.25,2.2);
\draw[rounded corners] (-0.25,-0.2) rectangle (-1.75,2.2);

\node[scale=0.8] (X) at(-3,-0.5) {$X$};
\node[scale=0.8] (Y) at(-1,-0.5) {$Y$};
\node[scale=0.8] (b) at(-3.2,1.85) {$b$};
\filldraw[black] (-3,1.8) circle (0.75pt);
\fill[fill=red!60, opacity=0.5] (-3,1.8)--(-3.52,0.7)--(-2.48,0.7);

\draw[fill=white] (-3,0.5) circle (0.55);
\node[scale=0.8] (X_1) at(-3,0.5) {$X_1$};
\draw[arrow=2] (-3,1.05) -- (-3,1.8);

\fill[fill=blue!100, opacity=0.3] (-3,1.8)--(-1,1.8)--(-1.1,0.8);

\draw[fill=white] (-1,1.3) circle (0.5);
\filldraw[black] (-1,0.9) circle (0.75pt);
\node[scale=0.8] (d) at(-0.82,0.95) {$d$};
\draw[arrow=2] (-3,1.8) -- (-1.5,1.4);

\fill[fill=red!60, opacity=0.5] (-1,0.9)--(-0.77,1.4)--(-1.23,1.4);
\draw[fill=white] (-1,1.5) circle (0.25);

\draw[arrow=1.6] (-1,1.25) -- (-1,0.9);

\node[scale=0.8] (Yb) at(-1,1.95) {$Y(b)$};
\node[scale=0.8] (Y1) at(-1,1.5) {$Y_1$};

\end{tikzpicture}
\qquad
\begin{tikzpicture}[scale = 1.2, xscale=1.6, yscale=1.2]
\draw[rounded corners] (-3.75,-0.2) rectangle (-2.25,2.2);
\draw[rounded corners] (-0.25,-0.2) rectangle (-1.75,2.2);

\node[scale=0.8] (X) at(-3,-0.5) {$X_1$};
\node[scale=0.8] (Y) at(-1,-0.5) {$Y_1$};
\node[scale=0.8] (c) at(-0.8,1.85) {$c$};

 \fill[fill=red!60, opacity=0.5] (-1,1.8)--(-3,1.8)--(-2.8,0.858);

\filldraw[black] (-1,1.8) circle (0.75pt);

\draw[arrow=2] (-1,1.8) -- (-1,1.05);
\fill[fill=blue!100, opacity=0.3] (-1,1.8)--(-0.485,0.7)--(-1.515,0.7);
\draw[fill=white] (-1,0.5) circle (0.55);
\node[scale=0.8] (Y_2) at(-1,0.5) {$Y_2$};

\draw[fill=white] (-3,1.3) circle (0.5);
\filldraw[black] (-3,0.9) circle (0.75pt);
\node[scale=0.8] (d) at(-3.18,0.95) {$a$};
\draw[arrow=2] (-2.5,1.4) -- (-1,1.8);

\draw[arrow=1.6] (-3,0.9)-- (-3,1.25) ;
\fill[fill=blue!100, opacity=0.3] (-3,0.9)--(-3.23,1.4)--(-2.77,1.4);
\draw[fill=white] (-3,1.5) circle (0.25);

\node[scale=0.8] (Yb) at(-3,1.95) {$X_1(c)$};
\node[scale=0.8] (Y1) at(-3,1.5) {$X_2$};

\end{tikzpicture}
\caption{The construction of $X_1,Y_1,X_2,Y_2$}
\end{figure}

\begin{proof}[ of Theorem \ref{bridgeslemma}]
Take $C_k = 2^{30k}$ for all even $k$ and $C_k = C_{k+1}$ for all odd $k$. Note that we only need to consider the case of $k$ even. Let $\delta$ be such that $\delta|X||Y| = \overrightarrow{e}(X,Y) \geq Cn^{2-\frac{2}{k}}$. By Claim \ref{generalbridgesclaim}, it is enough to show that $\delta^{\frac{k}{2}}|X|, \delta^{\frac{k}{2}}|Y| > 2^{10k^2}$. Indeed, note that
$$\delta^{\frac{k}{2}}|X| \geq \frac{C^{\frac{k}{2}} n^{k-1}}{|X|^{\frac{k}{2}}|Y|^{\frac{k}{2}}} \cdot |X| \geq C^{\frac{k}{2}} > 2^{10k^2}$$
since $|X|,|Y| \leq n$. Similarly, $\delta^{\frac{k}{2}}|Y| > 2^{10k^2}$.
\end{proof} 
\section{Proof of Theorem \ref{maintheorem}}\label{sec:proof of main result}
Let us recall the statement of Theorem \ref{maintheorem}. We will give its proof in this section. For a brief outline, we refer the reader to Section \ref{sec:proof overview}. 
\MainTheorem*
\begin{proof}
We will assume $c = c(k)$ is arbitrary large in terms of $k$ and will alert the reader to whenever this is most important. We also then let $n$ be sufficiently large.
First, we prepare our tournament in the next section. 
\subsection{Preparation}
Take a balanced partition $A', B'$ of $T$ which minimizes $\overrightarrow{e}(A',B')$. Note that we can assume that $\overrightarrow{e}(A',B') < n^{2-\frac{1}{1000k}} = o(n^2)$, as otherwise, we can apply Theorem \ref{thm:cut result} (with $\delta= 4n^{-\frac{1}{1000k}}$) to $T$ since $\delta^{0}(T) \geq \frac{n}{4}$ and get the desired $k$-th power of a Hamilton cycle. In turn, this gives the following.
\begin{lem}
There exist sets $A \subseteq A'$ and $B \subseteq B'$ such that both $T[A]$ and $T[B]$ are $o(n)$-almost regular and $R = V(T) \setminus (A \cup B)$ has size $o(n)$.
\end{lem}
\begin{proof}
Let $\delta > 0$ be such that $\overrightarrow{e}(A',B') = \delta n^2 /4$, so that $\delta = o(1)$. Define the sets $X_1 = \{v \in A' : d^{+}_{B'}(v) \geq 2 \sqrt{\delta} n\}$ and $Y_1 = \{v \in B': d^{-}_{A'}(v) \geq 2 \sqrt{\delta} n\}$. Note that $|X_1| \leq \frac{\sqrt{\delta}}{8} n$. Indeed, $\delta n^2/4 = \overrightarrow{e}(A',B') \geq 2\sqrt{\delta} n |X_1|$. Similarly, $|Y_1| \leq \frac{\sqrt{\delta}}{8} n$. Since $\delta^0(T) \geq \frac{n}{4}$, every vertex $v \in A' \setminus X_1$ is such that $d^{+}_{A'}(v) \geq \frac{n}{4} - 2 \sqrt{\delta} n$ and every vertex $v \in B' \setminus Y_1$ is such that $d^{-}_{B'}(v) \geq \frac{n}{4} - 2 \sqrt{\delta} n$.

Define now the sets $X_2 = \{v \in A': d^{+}_{A'}(v) \geq \frac{n}{4} + \delta^{\frac{1}{4}} n\}$ and $Y_2 = \{v \in B': d^{-}_{B'}(v) \geq \frac{n}{4} + \delta^{\frac{1}{4}} n\}$. Then, double-counting gives
$${|A'| \choose 2} = \sum_{v \in A'} d^{+}_{A'}(v) \geq |X_2| \left(\frac{n}{4} + \delta^{\frac{1}{4}} n \right) + (|A'| - |X_1| - |X_2|) \left(\frac{n}{4} - 2 \sqrt{\delta} n \right)$$ 
$$ \geq |X_2| \delta^{\frac{1}{4}} n + |A'| \left(\frac{n}{4} - 2 \sqrt{\delta} n \right) - |X_1| n .$$
Since $|A'| \leq \lceil \frac{n}{2} \rceil$ and $|X_1| \leq \frac{\sqrt{\delta}}{8}n$, this implies that
$$|X_2| \delta^{\frac{1}{4}} n \leq {|A'| \choose 2} - |A'|\left(\frac{n}{4} - 2 \sqrt{\delta}n \right) + |X_1|n \leq 5 \sqrt{\delta} n^2$$
and thus, $|X_2| \leq 5 \delta^{\frac{1}{4}}n$. Similarly, $|Y_2| \leq 5 \delta^{\frac{1}{4}}n$. To finish, take $A = A' \setminus (X_1 \cup X_2)$ and $B = B' \setminus (Y_1 \cup Y_2)$. Indeed, since $\delta = o(1)$, we have that $|R| = o(n)$. Further, every vertex $v \in A$ has $d^{+}_{A'}(v) = \frac{n}{4} + o(n)$ and so, $d^{+}_{A}(v) = d^{+}_{A'}(v) \pm |R| = \frac{n}{4} + o(n)$. The equivalent happens to $B$ and thus, $T[A]$ and $T[B]$ are both $o(n)$-almost regular. 
\end{proof}

Now, let us define the set $R_{\text{bad}} = \{v \in R: d^{-}_A(v) < \frac{1}{10} |A| \text{ and } d^{+}_B(v) < \frac{1}{10} |B|\}$ and let $r_{\text{bad}}$ denote its cardinality. Partition $R \setminus R_{\text{bad}}$ into two sets $R_A$ and $R_B$ so that $R_{A} \subseteq \{v \in R: d^{-}_A(v) \geq \frac{1}{10} |A|\}$ and $R_{B} \subseteq \{v \in R: d^{+}_B(v) \geq \frac{1}{10} |B|\}$. Let $r_A$, $r_B$ denote their cardinality. 
\subsubsection{Further preparation of $R_{bad}$, $R_A$ and $R_B$} 
By definition, every vertex in $R_{\text{bad}}$ has at least $\frac{9}{10}|A|$ out-neighbors in $A$ and at least $\frac{9}{10}|B|$ in-neighbors in $B$. Therefore, we can use Lemma \ref{lem:partition lemma} to get a collection of $O(\log r_{\text{bad}})$ vertex disjoint $(B,A,k)$-chains in $R_{\text{bad}}$ which cover all but at most $O(1)$ vertices. We will let $\mathcal{C}_{\text{bad}}$ denote this collection of chains and let $R'_{\text{bad}} \subseteq R_{\text{bad}}$ denote the set of vertices not covered by $\mathcal{C}_{\text{bad}}$.

Similarly, we can also partition $R_A$ and $R_B$ into chains and a small set. Indeed, note that every vertex in $R_A$ has at least $\frac{1}{10}|A|$ in-neighbors in $A$ and at least $\delta^0(T) - |R| \geq \frac{n}{5}$ out-neighbors in $A \cup B$. Therefore, Lemma \ref{lem:partition lemma} gives a collection of $O(\log r_A)$ vertex-disjoint $(A,A \cup B,10k)$-chains in $R_A$ which cover all but at most $O(1)$ vertices. This collection will be denoted by $\mathcal{C}_A$ and the set of uncovered vertices by $R'_A \subseteq R_A$. Equivalently, we also get a collection $\mathcal{C}_B$ of $O(\log r_B)$ vertex-disjoint $(A \cup B,B,10k)$-chains in $R_B$ which cover all but a set $R'_B \subseteq R_B$ of size $O(1)$.

Let $\pi$ denote the ordering of $V(\mathcal{C}_A) \cup V(\mathcal{C}_B)$ where the chains in $\mathcal{C}_A \cup \mathcal{C}_B$ are put consecutively in some arbitrary way and inherit their own ordering. Finally, it will be convenient to define a large integer $d =2^{200k}( |R'_A \cup R'_B|+1)= O(1)$ and consequently consider the set $R_{\text{good}} = \{v \in R'_A \cup R'_B : d^{+}_{A \cup R_A}(v) + d^{-}_{B \cup R_B}(v) \leq 2d\} \subseteq R'_A \cup R'_B$. Let $r_{\text{good}}$ denote its cardinality. Let us also define the sets $X = A \cup (R_A \setminus R_{\text{good}})$ and $Y = B \cup (R_B  \setminus R_{\text{good}})$.

\begin{figure}[t]
\begin{tikzpicture}[scale = 0.95, xscale=1, yscale=1]

\draw (-3,1) ellipse (2.5 and 2); 
\draw (3,1) ellipse (2.5 and 2); 

\draw (0,-3) circle (0.5);

\node[scale=1] (A) at(-6.5,1) {$A$};
\node[scale=1] (rA) at(-7,-1.75) {$R_A \setminus R_{\text{good}} $};
\node[scale=1] (B) at(6.5,1) {$B$};
\node[scale=1] (RB) at(7,-1.75) {$R_B \setminus R_{\text{good}} $};
\node[scale=0.8] (Rg) at(0,-3) {$R_{\text{good}}$};
\node[scale=1] (Rg) at(0,5) {$R_{\text{bad}}$};
\node[scale=0.8] (Rg) at(1.5,4) {$R'_{\text{bad}}$};
\node[scale=0.7] (R'g) at(-1.25,-1.75) {$R'_A \setminus R_{\text{good}} $};
\node[scale=0.7] (R'g) at(1.25,-1.75) {$R'_B \setminus R_{\text{good}} $};

\draw[rounded corners] (-2,3.5) rectangle (2,4.5);
\draw[rounded corners] (-5.5,-2.25) rectangle (-0.5,-1.25);
\draw[rounded corners] (5.5,-2.25) rectangle (0.5,-1.25);
\draw (-2,-2.25) -- (-2,-1.25);
\draw (2,-2.25) -- (2,-1.25);
\draw (1,3.5) -- (1,4.5);

\draw (-1.5,4.1) rectangle (0.8,4.3);

\fill[fill=red!30] (-1.8,3.9) -- (-1.5,3.9) -- (-1.5,3.7) -- (-2.96, 2.1) -- (-4.25,2.2);
\fill[white] (-3.65,2.1) ellipse (0.65 and 0.2);
\draw (-3.65,2.1) ellipse (0.65 and 0.2);
\draw (-1.8,3.7) rectangle (-1.5,3.9);
\draw[arow=2]  (-2,3.5) -- (-3.2,2.5);

\fill[fill=blue!30] (0.2,3.7) -- (0.2,3.9) -- (0.5,3.9) -- (4.25,2.2) -- (2.96, 2.1);
\fill[white] (3.65,2.1) ellipse (0.65 and 0.2);
\draw (3.65,2.1) ellipse (0.65 and 0.2);
\draw (0.2,3.7) rectangle (0.5,3.9);
\draw[arow=2]  (2.9,2.5) -- (0.9,3.5);

\draw[arrow=1.6]  (0.6,4.2) -- (-1.3,4.2);
\draw[arrow=1.6]  (0.3,3.8) -- (-1.6,3.8);

\draw (-1.8,3.7) rectangle (0.5,3.9);

\draw (-5.25,-2.05) rectangle (-2.6,-1.85);

\fill[fill=red!30] (-4.9,-1.45) -- (-4.9,-1.65) -- (-4.6,-1.65) -- (-4.6,-1.45) -- (-3.65, 0.5) -- (-4.95,0.5);
\fill[white] (-4.3,0.5) ellipse (0.65 and 0.2);
\draw (-4.3,0.5) ellipse (0.65 and 0.2);
\draw (-4.9,-1.65) rectangle (-4.6,-1.45);
\draw[arow=2]  (-4.4,0) -- (-4.7,-1.25);

\fill[fill=green!30] (-2.55,-1.45) -- (-2.55,-1.65) -- (-2.25,-1.65) -- (-2.25,-1.45) -- (-2.05, -0.4) --  (-3.35,-0.4);
\fill[white] (-2.7,-0.4) ellipse (0.65 and 0.2);
\draw (-2.7,-0.4) ellipse (0.65 and 0.2);

\draw[arow=2]  (-2.45,-1.3) -- (-2.6,-0.75);

\fill[fill=green!30] (-2.55,-1.45) -- (-2.55,-1.65) -- (-2.25,-1.65) -- (2.2, 1.4) --  (0.85,1.4);
\fill[white] (1.5,1.4) ellipse (0.65 and 0.2);
\draw (1.5,1.4) ellipse (0.65 and 0.2);

\draw (-2.55,-1.65) rectangle (-2.25,-1.45);
\draw (-4.9,-1.65) rectangle (-2.25,-1.45);

\draw[arow=2] (-1.55,-0.9) -- (0.65,0.8);

\draw[arrow=1.6]  (-4.7,-1.55) -- (-2.45,-1.55);
\draw[arrow=1.6]  (-5.05,-1.95) -- (-2.8,-1.95);
\draw (5.25,-2.05) rectangle (2.6,-1.85);
\draw[arrow=1.6]  (2.8,-1.95) -- (5.05,-1.95);

\fill[fill=blue!30] (4.9,-1.45) -- (4.9,-1.65) -- (4.6,-1.65) -- (4.6,-1.45) -- (3.65, 0.5) -- (4.95,0.5);
\fill[white] (4.3,0.5) ellipse (0.65 and 0.2);
\draw (4.3,0.5) ellipse (0.65 and 0.2);

\draw[arow=2]  (4.7,-1.25) -- (4.4,0);

\fill[fill=green!30] (2.55,-1.45) -- (2.55,-1.65) -- (2.25,-1.65) -- (2.25,-1.45) -- (2.05, -0.4) --  (3.35,-0.4);
\fill[white] (2.7,-0.4) ellipse (0.65 and 0.2);
\draw (2.7,-0.4) ellipse (0.65 and 0.2);
\draw (2.55,-1.65) rectangle (2.25,-1.45);
\draw[arow=2]  (2.6,-0.75) -- (2.45,-1.3);

\fill[fill=green!30] (2.55,-1.45) -- (2.55,-1.65) -- (2.25,-1.65) -- (-2.2, 1.4) --  (-0.85,1.4);
\fill[white] (-1.5,1.4) ellipse (0.65 and 0.2);
\draw (-1.5,1.4) ellipse (0.65 and 0.2);

\draw[arrow=1.6]  (2.45,-1.55) -- (4.7,-1.55);
\draw[arow=2] (-0.65,0.8) -- (1.55,-0.9);

\draw (4.9,-1.65) rectangle (4.6,-1.45);
\draw (4.9,-1.65) rectangle (2.25,-1.45);

\draw [decorate,decoration={brace,  amplitude=5pt,raise=2pt},yshift=0pt]($(-8.2,-2.5)$) -- ($(-8.2,3)$) node [black,midway,xshift=-0.8cm] {$X$};
\draw [decorate,decoration={brace,  amplitude=5pt,raise=2pt},yshift=0pt]($(8.2,3)$) -- ($(8.2,-2.5)$) node [black,midway,xshift=0.8cm] {$Y$};
\end{tikzpicture}

\caption{The partition of $R$ - the red regions denote common neighborhoods in $A$, the blue regions denote common neighborhoods in $B$ and the green regions denote common neighborhoods in $A \cup B$.}
\label{partition}
\end{figure}
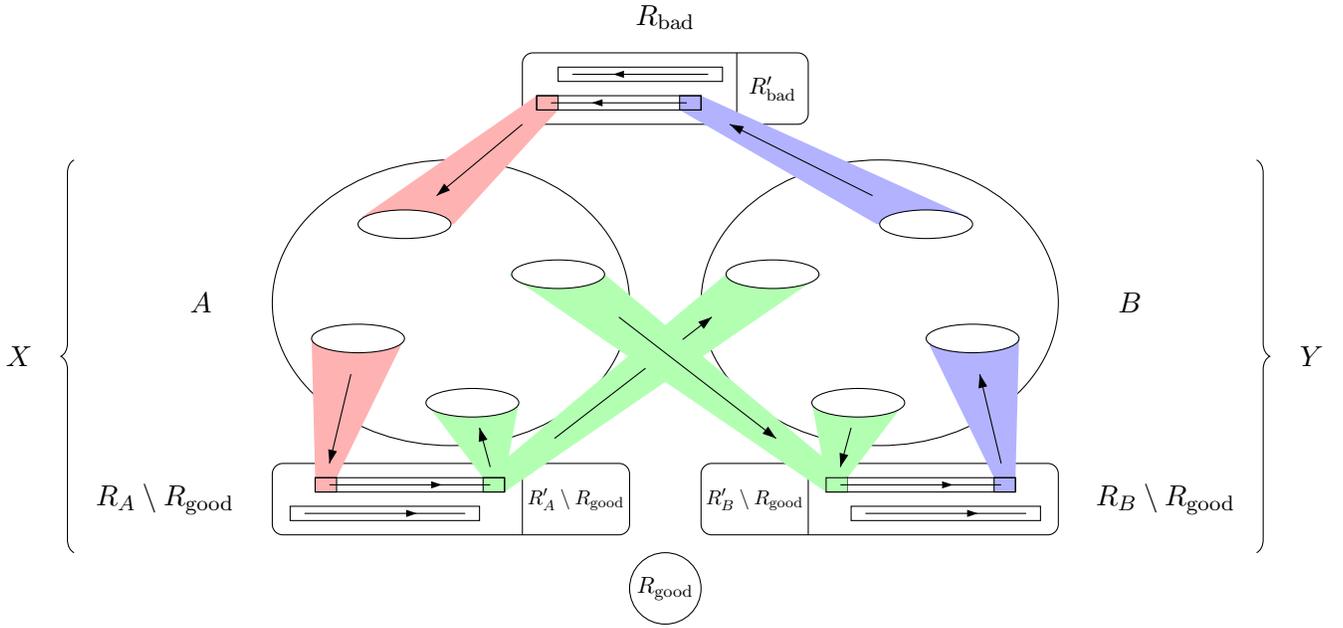
\subsection{Constructing bridges}\label{constructing}
As discussed in the proof outline, we will dedicate most of the proof to constructing bridges. First, let us define them. 
\begin{defn}
For $Z,W \in \{A,B\}$, a \emph{bridge going from $Z$ to $W$} is a $(Z,W,k)$-chain such that the following holds: its first $k$ vertices are in $Z$; its last $k$ vertices are in $W$; it intersects $A \cup B$ in at most $4k$ vertices.
\end{defn}
\noindent With this definition, we can now give a more detailed explanation of the proof strategy. 

\vspace{0.2cm}
\noindent
{\bf Main idea}
\vspace{0.2cm}

\noindent The goal is to find a collection $\mathcal{B}$ of $o(n)$ many vertex-disjoint bridges which cover $R$ (that is, every vertex in $R$ must be contained in some element of $\mathcal{B}$), such that the number of bridges in $\mathcal{B}$ going from $A$ to $B$ is positive and equal to the number of them going from $B$ to $A$. Now, recall first that both tournaments $T[A]$ and $T[B]$ are $o(n)$-almost regular and thus, so will be the tournaments $T[A \setminus V(\mathcal{B})]$ and $T[B \setminus V(\mathcal{B})]$, since $|\mathcal{B}| = o(n)$ and each bridge intersects $A \cup B$ in $O(1)$ many vertices. This also implies that they are both $\frac{1}{4}$-cut-dense and thus, they satisfy the conditions of Lemma \ref{lem:linking} and Theorem \ref{thm:cut result}. In particular, by Remark \ref{rem:hamiltonian path}, we can then find spanning chains $C_A \subseteq T[A \setminus V(\mathcal{B})]$ and $C_B \subseteq T[B \setminus V(\mathcal{B})]$. Afterwards, we can iteratively use Lemma \ref{lem:linking} (much like in Section \ref{cutdense2}) to link these spanning chains to the bridges in $\mathcal{B}$. Note indeed that the 'balanced' property of $\mathcal{B}$, that is, the fact that it contains as many bridges going from $A$ to $B$ as those going from $B$ to $A$, and further, that this number is positive, will allow us to construct these links in such a way that produces the $k$-th power of a Hamilton cycle. For an illustration of this fact, we refer the reader to Figure \ref{links}.

\vspace{0.2cm}
\noindent
{\bf The construction of $\mathcal{B}$}
\vspace{0.2cm}

\noindent In order to construct the collection $\mathcal{B}$, we will use the partitioning of $R$ that was done in the previous section (see Figure \ref{partition}). In effect, each part of $R$ will have to be covered by bridges in a specific manner. Moreover, at the same time, we will need to ensure that $\mathcal{B}$ has no imbalance, that is, that the number of bridges going from $A$ to $B$ is positive and equal to the number of them going from $B$ to $A$.

Now, in Section \ref{rarb}, we will start by finding $O(1)$ many vertex-disjoint bridges contained in $X \cup Y$ which cover $(R'_A \cup R'_B) \setminus R_{good}$ - these bridges will go either from $A$ to $A$, $A$ to $B$ or $B$ to $B$, have size at most $3k+1$ and form a collection which is $10k$-apart in $\pi$. Next, in Section \ref{precorr}, we find $r_{\text{good}} + |\mathcal{C}_{\text{bad}}| + |R'_{\text{bad}}| + 1 = O(\log n)$ many vertex-disjoint bridges (and also disjoint to the previous bridges) contained in $X \cup Y \cup R_{\text{good}}$ which cover $R_{\text{good}}$ (although not all of them have to intersect this set) - these will all go from $A$ to $B$, have size at most $4k+1$ and form, together with the previous bridges, a collection which is $10k$-apart in $\pi$. 

The two previous collections of bridges mentioned will be the 'hardest' ones to find. Afterwards, in Section \ref{complete}, we finish the construction of $\mathcal{B}$. First, note that since the previously found bridges are $10k$-apart in $\pi$ and have all size at most $4k+1$, we can use Lemma \ref{lem:destroying powers} to guarantee that each $C \in \mathcal{C}_A \cup \mathcal{C}_B$ is such that if we delete from it all the vertices used in previous bridges, we are still left with the $k$-th power of a path. Moreover, this also guarantees that each of these $k$-th powers are such that their  first and last $k$-sets of vertices have large common neighborhoods in $A \cup B$ consistent with those of the initial chains $C$. This means that we will be able to extend them into $|\mathcal{C}_A| + |\mathcal{C}_B| = O(\log n)$ many vertex-disjoint bridges which do not go from $B$ to $A$ and cover $(R_A \setminus R'_A) \cup (R_B \setminus R'_B)$. Similarly, we can also extend the chains in $\mathcal{C}_{\text{bad}}$ into $|\mathcal{C}_{\text{bad}}| = O(\log n)$ many vertex-disjoint bridges which go from $B$ to $A$ and cover $R_{\text{bad}} \setminus R'_{\text{bad}}$. To finish the covering, we find $|R'_{\text{bad}}| = O(1)$ many vertex-disjoint bridges going from $B$ to $A$ which cover $R'_{\text{bad}}$. For this, we will use that every vertex in $R'_{\text{bad}}$ has at least $\frac{9}{10}|A|$ out-neighbors in $A$ and at least $\frac{9}{10}|B|$ in-neighbors in $B$, and that $\overrightarrow{e}(B,A) \geq |A||B| - \overrightarrow{e}(A',B') =  |A||B| - o(n^2)$.

Finally, note that the covering of $R$ is done, but we have some imbalance in the collection of present bridges. Indeed, we have exactly $|\mathcal{C}_{\text{bad}}|+|R'_{\text{bad}}|$ bridges which go from $B$ to $A$ but in turn, we have at least $r_{\text{good}} + |\mathcal{C}_{\text{bad}}| + |R'_{\text{bad}}| + 1$ (and thus, a positive number of them) and at most $O(\log n)$ many bridges which go from $A$ to $B$. We then correct this imbalance by again using that $\overrightarrow{e}(B,A) \geq |A||B| - o(n^2)$ to create the necessary number of bridges going from $B$ to $A$.
\vspace{0.2cm}

\noindent We now proceed with the proof. First, let us set $\mathcal{B} = \emptyset$. 
\subsubsection{Covering $R'_A \setminus R_{good}$ and $R'_B \setminus R_{good}$}\label{rarb}
Building on previously discussed ideas, we cover $(R'_A \cup R'_B) \setminus R_{good}$ as follows.
\begin{lem}
There exists a collection of $O(1)$ vertex-disjoint bridges contained in $X \cup Y$ which cover $(R'_A \cup R'_B) \setminus R_{good}$ and are such that the following holds: none of them goes from $B$ to $A$ or has size more than $3k+1$; the collection is $10k$-apart in $\pi$. 
\end{lem}
\begin{proof}
We will find the bridges one by one. For this, let us first set $F, V = (R'_A \cup R'_B) \setminus R_{\text{good}}$ and $\mathcal{B'} = \emptyset$ and repeat the following procedure: Take a vertex $v \in V$ and find a bridge $\mathbf{v}$ containing $v$ of size no larger than $3k+1$ which does not go from $B$ to $A$ and such that $\mathbf{v} \setminus \{v\} \subseteq (X \cup Y) \setminus F$; add the vertices in $I_{10k}(\mathbf{v},\pi)$ to $F$; add $\mathbf{v}$ to $\mathcal{B'}$ and delete $v$ from $V$. We go until $V = \emptyset$. Note this indeed constructs the desired collection of bridges - recall that $|(R'_A \cup R'_B) \setminus R_{good}| = O(1)$.

We now show that the process can be done successfully. Consider some stage and let $v$ be the vertex taken. Suppose first that $v \in R'_A \setminus R_{\text{good}}$ - we will show that a bridge $\mathbf{v}$ can be constructed. Note that by the description of the process, we must have $|F| \leq (3k+1)(20k+1)|R'_A \cup R'_B| < \frac{1}{10}d$, since $d = 2^{200k}(|R'_A \cup R'_B|+1)$. Recall that our definitions imply that $d^{+}_{A \cup R_A}(v) + d^{-}_{B \cup R_B}(v) > 2d$ and $d^{-}_{A}(v) \geq |A|/10$. Let us then first assume that $d^{+}_{A \cup R_A}(v) > d$, which will be our first case. 

\textbf{Case 1:} \big($d^{+}_{A \cup R_A}(v) > d$\big)
Note this implies that $d^{+}_{(A \cup R_A) \setminus F}(v) \geq \frac{1}{2}d$ and so, we consider two subcases: either $d^{+}_{A \setminus F}(v) > \frac{1}{4}d$, or $d^{+}_{A \setminus F}(v) \leq \frac{1}{4}d$ and $d^{+}_{R_A \setminus F}(v) > \frac{1}{4}d$.

\textit{Case 1a:} \big($d^{+}_{A \setminus F}(v) > \frac{1}{4}d$\big) Denote $N^{+}_{A \setminus F}(v)$ and $N^{-}_{A \setminus F}(v)$ by $N^{+}$ and $N^{-}$, respectively. Recall that $T[A]$ is $o(n)$-almost regular and since $|F| = O(1)$, so is $T[A \setminus F]$. This implies by Lemma \ref{lem:anycutisdense} (with $\delta=1/4$), that $\overrightarrow{e}(N^{-},N^{+}) \geq \frac{1}{80}|N^{-}||N^{+}|$. We further know that $|N^{-}| \geq d^{-}_{A}(v) - |F| \geq \frac{|A|}{20}$ and $|N^{+}| \geq \frac{d}{4}$. We can then use Lemma \ref{DRC for sets} to get the following claim.
\begin{claim}\label{bridgeDRC}
There exists a subset $U \subseteq N^{+}$ of size at least $2^{30k}$ such that less than a fraction of $1/2^{30k^2}$ of its $k$-subsets have less than $2^{30k}$ common in-neighbors in $N^{-}$.
\end{claim}
\begin{proof}
Apply Lemma \ref{DRC for sets} to the graph on $N_{+} \cup N_{-}$ where the edges are those in $\overrightarrow{E}_{T}(N_{-},N_{+})$. We then take the following parameters: $a = |N_{+}| \in \left[\frac{d}{4},|A| \right]$, $b = |N_{-}| \in \left[\frac{1}{20}|A|,|A| \right]$, $\gamma \geq \frac{1}{80}$, $m = s = 2^{30k}$. Indeed, since $d \geq 2^{200k}$, note that
$${a \gamma^k \choose k} \geq \left( \frac{d}{4k 80^k}\right)^k > 2 \cdot 2^{60k^2} \cdot 20^k \geq 2 \cdot 2^{30k^2} \cdot {|A| \choose k}\left(\frac{2^{30k}}{\frac{1}{20}|A|} \right)^k \geq 2s^k\binom{a}{k}\left(\frac{m}{b}\right)^k$$
and also,
$${a \gamma^k \choose k} \geq 2^{60k^2} \geq 2s^k$$
\end{proof}

By Lemma \ref{inclusion-exlusion}, since $U \subseteq A$, the fraction of $A$-heads of size $k$ in it is at least $1/2^{30k^2}$. Thus, the above claim implies that there is an $A$-head in $U \subseteq N^{+}$, which has at least $2^{30k}$ common in-neighbors in $N^{-} \subseteq A$. In turn, we can apply Lemma \ref{inclusion-exlusion} again to ensure that there is an $A$-tail among these at least $2^{30k}$ vertices. Note that by construction, these two sets (the above found $A$-tail and $A$-head) together with $v$ give a bridge $\mathbf{v}$ as desired - it has size $2k+1$ and goes from $A$ to $A$.

\textit{Case 1b:} \big($d^{+}_{A \setminus F}(v) \leq \frac{1}{4}d$ and $d^{+}_{R_A \setminus F}(v) > \frac{1}{4}d$\big) Denote $N^{-}_{A \setminus F}(v)$ and $N^{+}_{R_A \setminus F} (v)$ by $N^{-}$ and $N^{+}$, respectively. Note that $|N^{-}| = |A \setminus F| - d^{+}_{A \setminus F}(v) \geq |A \setminus F| - \frac{d}{4} \geq |A| - \frac{d}{2}$ \big(recall that $|F| < \frac{1}{10}d$\big) and $|N^{+}| \geq \frac{d}{4}$. By definition, every vertex in $N^{+} \subseteq R_A$ has at least $\frac{|A|}{10}$ in-neighbors in $A$ and thus, at least $\frac{|N^{-}|}{12}$ in-neighbors in $N^{-}$ (by the previous inequality on $|N^{-}|$). Hence, $\overrightarrow{e}(N^{-},N^{+}) \geq \frac{1}{12}|N^{-}||N^{+}|$. A similar application of Lemma \ref{DRC for sets} to the one in the case before gives then that there is a subset $U \subseteq N^{+}$ of size at least $2^{30k}$ such that less than a fraction of $1/2^{30k^2}$ of its $k$-subsets have common in-neighborhood in $N^{-}$ of size less than $2^{30k}$. Further, recall that every vertex has at least $\frac{n}{5}$ out-neighbors in $A \cup B$. So, Lemma \ref{inclusion-exlusion} implies that the fraction of $(A \cup B)$-heads of size $k$ in $U$ is at least $1/2^{30k^2}$ and thus, there is a set $S_2 \subseteq U$ of size $k$ which is a $(A \cup B)$-head and has an $A$-tail $S_1 \subseteq N^{-}$ of size $k$ in its common in-neighborhood in $N^{-}$. Finally, since $S_2$ is an $(A \cup B)$-head, its common out-neighborhood in $(A \cup B) \setminus F$ is of size $\Omega(n)$ and so, there is a set $S_3$ of size $k$ lying in its common out-neighborhood which is an $A$-head contained in $A \setminus F$ or a $B$-head contained in $B \setminus F$. Note that $\mathbf{v} = (S_1, v, S_2, S_3)$ is a bridge as desired - it has size $3k+1$ and goes from $A$ to $A$ or from $A$ to $B$.

\textbf{Case 2:} \big($d^{+}_{A \cup R_A}(v) \leq d$\big) Recall that we have $d^{+}_{A \cup R_A}(v) + d^{-}_{B \cup R_B}(v) > 2d$ since $v \notin R_{\text{good}}$. Thus, we must have $d^{-}_{B \cup R_B}(v) \geq d$ which implies $d^{-}_{(B \cup R_B) \setminus F}(v) \geq \frac{1}{2}d$. Note also, that since $N^{+}(v) \geq \delta^0(T) \geq \frac{n}{4}$, we have $d^{+}_{B \setminus F}(v) \geq N^{+}(v) - d^{+}_{A \cup R_A}(v) - |R| \geq \frac{n}{5}$ (note that $F \subseteq R$). We can then find $\mathbf{v}$ by breaking the situation into two subcases which are very similar to before: either $d^{-}_{B \setminus F}(v) > \frac{1}{4}d$ or $d^{-}_{B \setminus F}(v) \leq \frac{1}{4}d$  and $d^{-}_{R_B \setminus F}(v) > \frac{1}{4}d$. By doing the same procedures as before, the first case will give a bridge of size $2k+1$ which goes from $B$ to $B$, while the second case will give a bridge of size $3k+1$ going from $B$ to $B$ or $A$ to $B$.

Finally, the construction is analogous when $v \in R'_B \setminus R_{\text{good}}$.
\end{proof}
Add the bridges given by the claim to $\mathcal{B}$, so that $|V(\mathcal{B})| = O(1)$ and $\mathcal{B}$ is $10k$-apart in $\pi$.
\subsubsection{Finding bridges which go from $A$ to $B$ and covering $R_{good}$}\label{precorr}
Every $v \in R_{\text{good}}$ is such that $d^{-}_X(v) \geq |X| - 2d$ and $d^{+}_Y(v) \geq |Y| - 2d$. Hence, since $d, |V(\mathcal{B})|, r_{good} = O(1)$, there exist subsets $X' \subseteq X \setminus V(\mathcal{B})$ and $Y' \subseteq Y \setminus V(\mathcal{B})$ of sizes $|X| - O(1)$ and $|Y| - O(1)$, such that $X' \rightarrow R_{\text{good}} \rightarrow Y' $. 

We will now get a density condition between $X'$ and $Y'$ that will be enough to construct many bridges going from $A$ to $B$. These will be used to cover $R_{good}$ and to balance out the number of bridges going from $B$ to $A$ which are created later. Recall that indeed this correction must be made because we will want to have in the end the same number of bridges going from $A$ to $B$ as those going from $B$ to $A$.

Now, let us first assume that $|X'| \leq |Y'|$, implying that $|X'| \leq \frac{1}{2}|X \cup Y| \leq \frac{1}{2} (n-r_{\text{bad}} - r_{\text{good}})$. Since $\delta^0(T) \geq \frac{n}{4} + cn^{1- 1/ \lceil k/2 \rceil}$, we have $$\overrightarrow{e}(X',V(T) \setminus X') \geq \sum_{v \in X'} d^{+}(v) - {|X'| \choose 2} \geq |X'| \left(\frac{n}{4} + cn^{1- 1/ \lceil k/2 \rceil} -  \frac{n-r_{\text{bad}} - r_{\text{good}}}{4}\right)$$ 
\begin{equation}\label{dens}
= |X'| \left(\frac{r_{\text{bad}} + r_{\text{good}}}{4} + cn^{1- 1/ \lceil k/2 \rceil} \right) .
\end{equation}
Moreover, every vertex $v \in R_{\text{bad}}$ is such that $d^{-}_X(v) \leq d^{-}_A(v) + |R| \leq \frac{1}{10}|A| +o(n) = \frac{1}{10}|X'| +o(n)$ and so, $\overrightarrow{e}(X',R_{\text{bad}}) \leq \left(\frac{|X'|}{10} +o(n) \right) r_{\text{bad}}$. Hence, as $r_{\text{good}} + |X \setminus X'| + |Y \setminus Y'|= O(1)$ and $\overrightarrow{e}(X',Y') = \overrightarrow{e}(X',V(T) \setminus X') - \overrightarrow{e}(X',R_{\text{bad}}) - \overrightarrow{e}(X',R_{\text{good}} \cup (X \setminus X') \cup (Y \setminus Y'))$, we can deduce from (\ref{dens}) that
\begin{equation} \label{densitycondition2}
\overrightarrow{e}(X',Y') \geq |X'| \left(\frac{r_{\text{bad}}}{8} + cn^{1- 1/ \lceil k/2 \rceil} - O(1) \right) \geq \frac{n}{3} \left(\frac{r_{\text{bad}}}{8} + cn^{1- 1/ \lceil k/2 \rceil} - O(1) \right) 
\end{equation}
since $|X'|,|Y'| \geq \frac{n}{3}$. Similarly, we can deduce the same inequality when $|Y'| \leq |X'|$. This density condition is used in the following.
\begin{lem}\label{claim1}
There exists a collection of $r_{\text{good}} + |\mathcal{C}_{\text{bad}}| + |R'_{\text{bad}}| + 1$ many vertex-disjoint bridges contained in $X \cup Y \cup R_{good}$ which cover $R_{good}$ and are such that the following hold: each of them goes from $A$ to $B$ and has size at most $4k+1$; the collection together with $\mathcal{B}$ is $10k$-apart in $\pi$. 
\end{lem}
\begin{proof}
We will first find a collection $\mathcal{B}'$ of $r_{\text{good}} + |\mathcal{C}_{\text{bad}}| + |R'_{\text{bad}}| + 1$ many vertex-disjoint bridges in $A^k \times (X')^{k} \times (Y')^{k} \times B^k$ which together with $\mathcal{B}$ is $10k$-apart in $\pi$. Later we will insert each vertex of $R_{\text{good}}$ into a unique bridge in $\mathcal{B}'$ to get the desired covering.

Set first $\mathcal{B}' = \emptyset$ and $F = V(\mathcal{B}) \cup I_{10k}(V(\mathcal{B}),\pi)$ and repeat the following operation: if possible, find a bridge $\mathbf{v} \in (A \setminus F)^k \times (X'\setminus F)^{k} \times (Y' \setminus F)^{k} \times (B \setminus F)^k $ going from $A$ to $B$ and add it to $\mathcal{B}'$; add all vertices in $\mathbf{v} \cup I_{10k}(\mathbf{v}, \pi)$ to $F$.

Consider some stage of the process. Note that $|X' \setminus F|,|Y' \setminus F| \geq \frac{n}{3} - |F|$ and that by (\ref{densitycondition2}),
$$\overrightarrow{e}(X' \setminus F,Y' \setminus F) \geq \frac{n}{3} \left(\frac{r_{\text{bad}}}{8} + cn^{1- 1/ \lceil k/2 \rceil} - O(1) \right) - n |F| \geq n \left(\frac{r_{\text{bad}}}{30} + \frac{c}{3}n^{1- 1/ \lceil k/2 \rceil} - |F| - O(1) \right) .$$
Further, all vertices $x \in X' \setminus F \subseteq A \cup R_A$ and $y \in Y' \setminus F \subseteq B \cup R_B$ are such that $d^{-}_{A \setminus F}(x) \geq \frac{|A|}{10} - |F|$ and $d^{+}_{B \setminus F}(y) \geq \frac{|B|}{10} - |F|$. Note also that $|F| \leq (|V(\mathcal{B})|+|V(\mathcal{B'})|)(20k+1) \leq (O(1) + 4k|\mathcal{B'}|)(20k+1) = O( |\mathcal{B'}| + 1)$. 

Therefore, if $|\mathcal{B'}| \leq r_{\text{good}} + |\mathcal{C}_{\text{bad}}| + |R'_{\text{bad}}| = O(\log r_{\text{bad}}) + O(1)$ (since $|R'_{\text{bad}}|, r_{\text{good}} = O(1)$), we have that $|F| \leq  \frac{r_{\text{bad}}}{30} + O(1)$ and thus,
\begin{equation}\label{density4}
\overrightarrow{e}(X' \setminus F,Y' \setminus F) \geq \frac{c}{3}n^{2- 1/ \lceil k/2 \rceil} - O(n) .
\end{equation}
Now, let us recall again that for all $x \in X' \setminus F$, we have $d^{-}_{A \setminus F}(x) \geq \frac{1}{10}|A| - |F| \geq \frac{1}{10}|A| - o(n) \geq \frac{1}{20}|A \setminus F|$ and similarly, for all $y \in Y' \setminus F$, we have $d^{+}_{B \setminus F}(y) \geq \frac{1}{20}|B \setminus F|$. Also, $|A \setminus F|, |B \setminus F| \geq \frac{n}{3}$. Now, choose independently and uniformly at random an $A$-tail $S_1 \subseteq A \setminus F$ of size $k$ and a $B$-head $S_2 \subseteq B \setminus F$ of size $k$. Define $X_1 = N^{+}_{X' \setminus F}(S_1)$ and $Y_1 = N^{-}_{Y' \setminus F}(S_2)$. We can show that then
\begin{equation}\label{density10}
\mathbb{E}[\overrightarrow{e}(X_1,Y_1)] \geq \overrightarrow{e}(X' \setminus F,Y' \setminus F)/2^{100k^2} .
\end{equation}
Indeed, take a vertex $x \in X' \setminus F$. Since $T[A]$ is $o(n)$-almost regular, every vertex in $N^{-}_{A \setminus F}(x) \subseteq A$ has at least $\frac{|A|}{4}$ in-neighbors in $A$. Therefore, Lemma \ref{inclusion-exlusion} gives that the fraction of $A$-tails in $N^{-}_{A \setminus F}(x)$ is at least $1/2^{30k^2}$ and consequently, the number of such $A$-tails is at least $(d^{-}_{A \setminus F}(x))^k /2^{30k^2} \geq |A \setminus F|^k/2^{50k^2}$. Hence, $\mathbb{P}(x \in X_1) \geq 1/2^{50k^2}$. Similarly, for all $y \in Y' \setminus F$, $\mathbb{P}(y \in Y_1) \geq 1/2^{50k^2}$, which ultimately, gives (\ref{density10}) by summing over all edges.

Take then a choice of $S_1,S_2$ such that $\overrightarrow{e}(X_1,Y_1) \geq \overrightarrow{e}(X' \setminus F,Y' \setminus F)/2^{100k^2}$. By (\ref{density4}), if $c$ is sufficiently large as a function of $k$, we can apply Theorem \ref{bridgeslemma} to the tournament $T[X_1 \cup Y_1]$. Let $ \mathbf{w}_1 \times \mathbf{w}_2 \in X_1^k \times Y_1^k$ be the resulting $k$-th power. Then, $\mathbf{v} = (S_1,\mathbf{w}_1,\mathbf{w}_2,S_2) \in (A \setminus F)^k \times X_1^{k} \times Y_1^{k} \times (B \setminus F)^k $ is the desired bridge which we add to $\mathcal{B'}$. The described process can then be successfully repeated until $|\mathcal{B'}| = r_{\text{good}} + |\mathcal{C}_{\text{bad}}| + |R'_{\text{bad}}| + 1$. 

To conclude, for each vertex in $R_{\text{good}}$ we arbitrarily choose a unique bridge in $\mathcal{B'}$, to which we add the vertex - making it a bridge in $A^k \times (X')^{k} \times R_{\text{good}} \times (Y')^{k} \times B^k$. This works because $X' \rightarrow R_{\text{good}} \rightarrow Y' $. The new collection $\mathcal{B'}$ of bridges will now cover $R_{\text{good}}$ and have the desired properties.
\end{proof}
Add the bridges given by the claim to $\mathcal{B}$. Note that $|\mathcal{B}| = O(\log r_{\text{bad}}) + O(1) = O(\log n)$ and that $\mathcal{B}$ is $10k$-apart in $\pi$.
\subsubsection{Completing the covering}\label{complete}
In this section, we will complete the covering of $R$ and consequent construction of $\mathcal{B}$, by extending what is left of the chains in $\mathcal{C}_A$, $\mathcal{C}_B$ and $\mathcal{C}_{\text{bad}}$ and correcting whatever imbalance we have as a result. Indeed, recall that we want the number of bridges going from $A$ to $B$ to be positive and equal to the number of bridges going from $B$ to $A$. Before doing so, we define what we mean by extending. If $Q$ is the $k$-th power of a path, then an \emph{extension of $Q$} is a bridge $\mathbf{v} = \mathbf{v}_1 \times Q \times \mathbf{v}_2$, where $\mathbf{v}_1,\mathbf{v}_2$ have both $k$ vertices.

First, note that by construction, the collection $\mathcal{B}$ is $10k$-apart in $\pi$. Note also that every bridge in $\mathcal{B}$ at the moment has size at most $4k+1$. Therefore, Lemma \ref{lem:destroying powers} implies that each $C \in \mathcal{C}_A \cup \mathcal{C}_B$ is such that $C \setminus V(\mathcal{B})$ is the $k$-th power of a path. Further, for all $10k$-chains $C \in \mathcal{C}_A$, the first $k$ vertices of $C \setminus V(\mathcal{B})$ are among the first $10k$ vertices of $C$ and thus have $\Omega(n)$ common in-neighbors in $A$. Similarly, the last $k$ vertices of $C \setminus V(\mathcal{B})$ have $\Omega(n)$ common out-neighbors in $A \cup B$. Finally, if $C \in \mathcal{C}_B$, the first $k$ vertices of $C \setminus V(\mathcal{B})$ have $\Omega(n)$ common in-neighbors in $A \cup B$ and its last $k$ have $\Omega(n)$ common out-neighbors in $B$.

Since $|\mathcal{B}| = O(\log n)$ and $|\mathcal{C}_A \cup \mathcal{C}_B \cup \mathcal{C}_{\text{bad}}| = O(\log n)$, it is easy to see that we can then get a collection $\{\mathbf{v}_C : C \in \mathcal{C}_A \cup \mathcal{C}_B \cup \mathcal{C}_{\text{bad}}\}$ of vertex-disjoint bridges such that for each $C$, $\mathbf{v}_C$ is an extension of $C \setminus V(\mathcal{B})$ and $\mathbf{v}_C$ is disjoint to $V(\mathcal{B})$. This is done by using Lemma \ref{inclusion-exlusion} in a similar way as in many previous arguments. Furthermore, if $C \in \mathcal{C}_{\text{bad}}$, then $\mathbf{v}_C$ is a bridge going from $B$ to $A$ and if $C \in \mathcal{C}_A \cup \mathcal{C}_B$, it is a bridge going from $A$ to $A$, $A$ to $B$ or $B$ to $B$. Add the bridges $\mathbf{v}_C$ to $\mathcal{B}$. Note that we added $O(\log n)$ bridges and so, we still have $|\mathcal{B}| = O(\log n)$.

To finish the covering, we are only left to deal with $R'_{\text{bad}}$. Recall that every vertex $v$ in this set is such that $d^{+}_A(v) \geq \frac{9}{10}|A|$ and $d^{-}_B(v) \geq \frac{9}{10}|B|$. Furthermore, recall also that $\overrightarrow{e}(B,A) = |A||B| - \overrightarrow{e}(A,B) \geq |A||B| - o(n^2)$ and that $|V(\mathcal{B}) \cap (A \cup B)| = O(\log n)$. Therefore, for a specific vertex $v \in R'_{\text{bad}}$, we can use Lemma \ref{DRC for sets} with density $\gamma = 1-o(1)$ to construct a bridge of size $2k+1$ which goes from $B$ to $A$ and contains $v$ as its middle vertex. Since $|R'_{\text{bad}}| = O(1)$, we can iterate this to find a collection of $|R'_{\text{bad}}|$ vertex-disjoint bridges of size $2k+1$ contained in $(A \cup B \cup R'_{\text{bad}}) \setminus V(\mathcal{B})$ which go from $B$ to $A$ and cover $R'_{\text{bad}}$. Add these bridges to $\mathcal{B}$ - we still have $|\mathcal{B}| = O(\log n)$.

To conclude, the covering of $R$ is done but there might still be some imbalance in $\mathcal{B}$. Indeed, note that the only construction of bridges going from $B$ to $A$ was done to cover $R_{\text{bad}}$ in the two previous paragraphs. We then have precisely $|R'_{\text{bad}}| + |\mathcal{C}_{\text{bad}}|$ bridges in $\mathcal{B}$ going from $B$ to $A$. On the other hand, by Lemma \ref{claim1}, we have at least $r_{\text{good}} + |\mathcal{C}_{\text{bad}}| + |R'_{\text{bad}}| + 1$ bridges in $\mathcal{B}$ which go from $A$ to $B$. Therefore, there is a number $t \leq |\mathcal{B}| = O(\log n)$, such that if $t$ bridges which go from $B$ to $A$ are added to $\mathcal{B}$, there is no imbalance. Since $\overrightarrow{e}(B,A) \geq |A||B| - o(n^2)$ and $|V(\mathcal{B}) \cap (A \cup B)| = O(\log n)$, we can indeed (again by using Lemma \ref{DRC for sets}) construct $t$ vertex-disjoint bridges of size $2k$ contained in $(A \cup B) \setminus V(\mathcal{B})$ which go from $B$ to $A$. We are then finished by adding these to $\mathcal{B}$.
\subsection{Finishing the proof}\label{finish}
\begin{figure}[t]
\begin{tikzpicture}[scale = 1, xscale=1, yscale=0.8]

\draw (-5,0) ellipse (3 and 4); 
\node[scale=1.2] (A) at(-5,5) {$A$};
\draw (5,0) ellipse (3 and 4); 
\node[scale=1.2] (B) at(5,5) {$B$};
\draw[rounded corners] (-1,-4) rectangle (1,4);
\node[scale=1.2] (R) at(0,5) {$R$};

\fill[ left color=gray,right color=white ] (-4,1.8) rectangle (4,2.2);
\draw (-4,1.8) rectangle (4,2.2);
\fill[left color=gray,right color=white] (-4,2.4) rectangle (4,2.8);
\draw (-4,2.4) rectangle (4,2.8);

\fill[left color=gray,right color=white] (-4,1.2) rectangle (4,1.6);
\draw (-4,1.2) rectangle (4,1.6);

\fill[right color=gray,left color=white] (-4,0.4) rectangle (4,0.8);
\draw (-4,0.4) rectangle (4,0.8);
\fill[right color=gray,left color=white] (-4,-0.2) rectangle (4,0.2);
\draw (-4,-0.2) rectangle (4,0.2);

\fill[right color=gray,left color=white] (-4,-0.8) rectangle (4,-0.4);
\draw (-4,-0.8) rectangle (4,-0.4);

\fill[left color=gray,right color=white] (-4,-1.6) -- (-0.4,-1.6) -- (-0.4,-2.8) -- (-4,-2.8) -- (-4,-2.4) -- (-0.8,-2.4) -- (-0.8,-2) -- (-4,-2);
\draw (-4,-1.6) -- (-0.4,-1.6);
\draw (-4,-2) -- (-0.8,-2);
\draw (-4,-2.4) -- (-0.8,-2.4);
\draw (-4,-2.8) -- (-0.4,-2.8);
\draw (-0.8,-2) -- (-0.8,-2.4);
\draw (-0.4,-1.6) -- (-0.4,-2.8);
\draw (-4,-1.6) -- (-4,-2);
\draw (-4,-2.4) -- (-4,-2.8);

\fill[right color=gray,left color=white] (4,-1.6) -- (0.4,-1.6) -- (0.4,-2.8) -- (4,-2.8) -- (4,-2.4) -- (0.8,-2.4) -- (0.8,-2) -- (4,-2);
\draw (4,-1.6) -- (0.4,-1.6);
\draw (4,-2) -- (0.8,-2);
\draw (4,-2.4) -- (0.8,-2.4);
\draw (4,-2.8) -- (0.4,-2.8);
\draw (0.8,-2) -- (0.8,-2.4);
\draw (0.4,-1.6) -- (0.4,-2.8);
\draw (4,-1.6) -- (4,-2);
\draw (4,-2.4) -- (4,-2.8);

\draw[arrow=2] (-3.6,2) -- (3.6,2);
\draw[arrow=2] (-3.6,2.6) -- (3.6,2.6);
\draw[arrow=2] (-3.6,1.4) -- (3.6,1.4);
\draw[arrow=2] (3.6,0.6) -- (-3.6,0.6);
\draw[arrow=2] (3.6,0) -- (-3.6,0);
\draw[arrow=2] (3.6,-0.6) -- (-3.6,-0.6);

\draw (-3.6,-1.8) -- (-0.6,-1.8);
\draw (-3.6,-2.6) -- (-0.6,-2.6);
\draw[arrow=2] (-0.6, -1.8) -- (-0.6,-2.6);
\draw (3.6,-1.8) -- (0.6,-1.8);
\draw (3.6,-2.6) -- (0.6,-2.6);
\draw[arrow=2] (0.6, -2.6) -- (0.6,-1.8);

\draw[thick,blue] (4,2.6) -- (4.4,2.6);
\draw[thick,blue] (4,0.6) -- (4.4,0.6);
\draw[thick,blue,arrow=3] (4.4, 2.6) -- (4.4,0.6);

\draw[thick,blue] (4,2) -- (4.8,2);
\draw[thick,blue] (4,0) -- (4.8,0);
\draw[thick,blue,arrow=3] (4.8, 2) -- (4.8,0);

\draw[thick,red] (-4,2) -- (-4.4,2);
\draw[thick,red] (-4,0.6) -- (-4.4,0.6);
\draw[thick,red,arrow=3] (-4.4, 0.6) -- (-4.4,2);

\draw[thick,red] (-4,1.4) -- (-4.8,1.4);
\draw[thick,red] (-4,0) -- (-4.8,0);
\draw[thick,red,arrow=3] (-4.8, 0) -- (-4.8,1.4);

\draw[thick,red] (-4,-0.6) -- (-4.4,-0.6);
\draw[thick,red] (-4,-1.8) -- (-4.4,-1.8);
\draw[thick,red,arrow=3] (-4.4, -0.6) -- (-4.4,-1.8);

\fill[red!30] (-4.8,-3) -- (-6.6,-3) -- (-6.6,3) -- (-4.8,3) -- (-4.8,2.2) -- (-5.6,2.2) -- (-5.6,-2.2) -- (-4.8,-2.2) -- (-4.8,-3);
\draw[arrow=3] (-6.1,-2.6) -- (-6.1,2.6);
\draw (-6.1,-2.6) -- (-5.2,-2.6);
\draw (-6.1,2.6) -- (-5.2,2.6);
\node[scale=1] (CA) at(-7.2,0) {$C_A$};

\fill[blue!30] (4.8,-3) -- (6.6,-3) -- (6.6,3) -- (4.8,3) -- (4.8,2.2) -- (5.6,2.2) -- (5.6,-2.2) -- (4.8,-2.2) -- (4.8,-3);
\draw[arrow=3] (6.1,2.6) -- (6.1,-2.6);
\draw (6.1,-2.6) -- (5.2,-2.6);
\draw (6.1,2.6) -- (5.2,2.6);
\node[scale=1] (CA) at(7.2,0) {$C_B$};

\draw[thick,red,arrow=3] (-4.8,2.6) -- (-4,2.6);
\draw[thick,red,arrow=3] (-4,-2.6) -- (-4.8,-2.6);

\draw[thick,blue,arrow=3] (4.8,-2.6) -- (4,-2.6);
\draw[thick,blue] (4,-1.8) -- (4.4,-1.8);
\draw[thick,blue] (4,-0.6) -- (4.4,-0.6);
\draw[thick,blue,arrow=3] (4.4, -1.8) -- (4.4,-0.6);

\draw[thick,blue,arrrow=3] (4,1.4) -- (4.8,2.6);

\end{tikzpicture}

\caption{The 'links' which produce the $k$-th power of a Hamilton cycle - when $l = 3$ and $l_A = l_B = 1$.}
\label{links}
\end{figure}
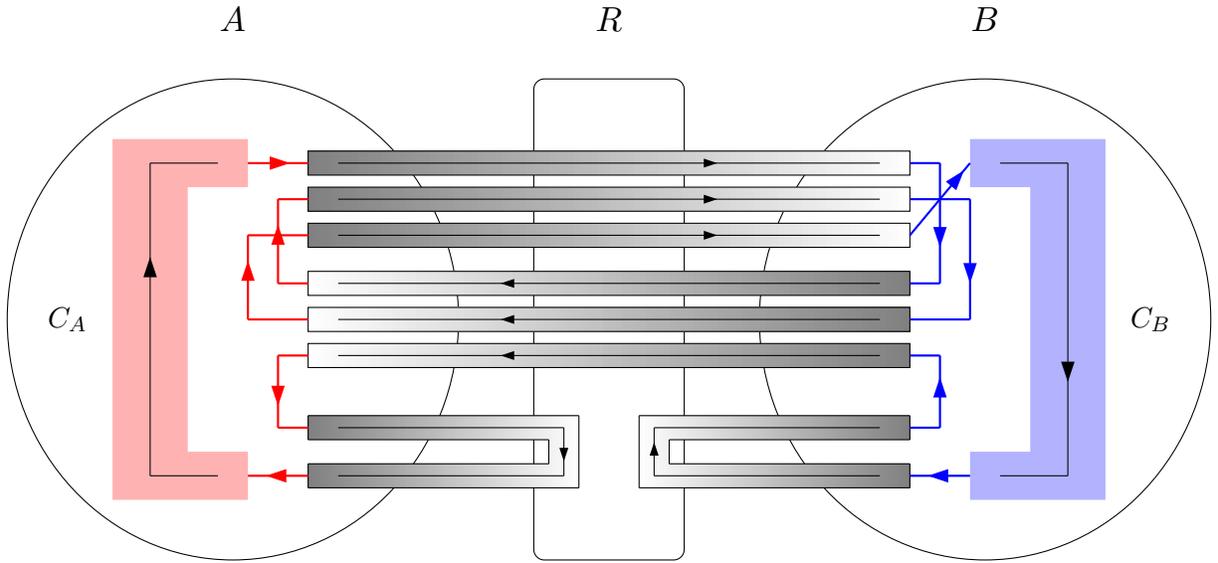

At this point, the situation is as follows. We have a collection $\mathcal{B}$ of vertex-disjoint bridges covering $R$, such that the number of them going from $A$ to $B$ is positive and equal to the number of them going from $B$ to $A$. Moreover, each bridge intersects $A \cup B$ in at most $4k$ vertices and $\mathcal{B}$ has size $O(\log n)$. Hence, both $T[A \setminus V(\mathcal{B})]$ and $T[B \setminus V(\mathcal{B})]$ are $o(n)$-almost regular. In particular, they are $\frac{1}{4}$-cut-dense and respectively have minimum semidegrees at least $|A|/4$ and $|B|/4$. Therefore, we can use Remark \ref{rem:hamiltonian path} to find spanning $3k$-chains $C_A \subseteq T[A \setminus V(\mathcal{B})]$ and $C_B \subseteq T[B \setminus V(\mathcal{B})]$.

We will now use Lemma \ref{lem:linking} inside the $o(n)$-almost regular tournaments $T[A \setminus V(\mathcal{B})]$ and $T[B \setminus V(\mathcal{B})]$ much like in Section \ref{cutdense2}, in order to link the chains $C_A,C_B$ with the bridges in $\mathcal{B}$ so that as a result, we get the $k$-th power of a Hamilton cycle in $T$. First, let us indicate which 'links' need to be made. For purposes of notation, let $\mathbf{v}_1, \ldots, \mathbf{v}_l \in \mathcal{B}$ denote the bridges going from $A$ to $B$, let $\mathbf{v}_{l+1}, \ldots, \mathbf{v}_{2l} \in \mathcal{B}$ denote the bridges going from $B$ to $A$, let $\mathbf{w}_{1}, \ldots, \mathbf{w}_{l_A} \in \mathcal{B}$ denote the bridges going from $A$ to $A$ and let $\mathbf{u}_{1}, \ldots, \mathbf{u}_{l_B} \in \mathcal{B}$ denote the bridges going from $B$ to $B$.

Inside $A$, we will make the following links: $C_A$ to $\mathbf{v}_1$ (meaning we link the last $k$ vertices of $C_A$ to the first $k$ vertices of $\mathbf{v}_1$), $\mathbf{v}_{l+i-1}$ to $\mathbf{v}_i$ (for $2 \leq i \leq l$), $\mathbf{v}_{2l}$ to $\mathbf{w}_1$, $\mathbf{w}_{i-1}$ to $\mathbf{w}_{i}$ (for $1 < i \leq l_A$) and $\mathbf{w}_{l_A}$ to $C_A$. Inside $B$: $\mathbf{v}_i$ to $\mathbf{v}_{l+i}$ (for $1 \leq i \leq l-1$), $\mathbf{v}_l$ to $C_B$, $C_B$ to $\mathbf{u}_{1}$, $\mathbf{u}_{i}$ to $\mathbf{u}_{i+1}$ (for $1 \leq i \leq l_B-1$), $\mathbf{u}_{l_B}$ to $\mathbf{v}_{2l}$. Note this will produce the $k$-th power of a Hamilton cycle (see Figure \ref{links} for an illustration).

As mentioned, the construction of these linking structures is done as in Section \ref{cutdense2}. Also, we will only do it for the links inside $A$, since for $B$ the method is the same. Now, first let $F$ be the set of vertices which are either in $V(\mathcal{B})$ or among the first or last $k$ vertices of $C_A$. Note that $|F| = O(\log n)$. Let also $V = A \setminus F$ and $\pi$ be the ordering of $V$ given by the chain $C_A$. Proceed with the first linking, i.e., $C_A$ to $\mathbf{v}_1$. Let $H_A$ denote the set of last $k$ vertices of $C_A$ and $H'_1$ denote the set of first $k$ vertices of $\mathbf{v}_1$. By the definition of a $3k$-chain as well as the definition of a bridge, we have that $H_A$ has at least $|A|/2^{20k}$ common out-neighbors in $V$ and $H'_1$ has at least $|A|/2^{20k}$ common in-neighbors in $V$. Therefore, we can apply Lemma \ref{lem:linking} to the tournament $T[V \cup H_A \cup H'_1]$, since it is $o(n)$-almost regular. This will produce a $k$-th power $P \subseteq V \cup H_A \cup H'_1$ which starts in $H_A$ and ends in $H'_1$ and further, is such that $P \setminus (H_A \cup H'_1)$ can be partitioned into at most $\log n$ many $k$-sets $S_1, \ldots, S_m$ which are $10k$-apart in $\pi$. Update $F:= F \cup I_{10k}(S_1 \cup \ldots \cup S_m,\pi)$ and $V := V \setminus F$. This implies that $V$ decreases by at most $O(\log n)$. We repeat this process until we've made all the links needed inside $A$. Notice that at each step, we delete $O(\log n)$ many vertices from $V$ and there are $O(\log n)$ many steps. Thus, at each step, the tournament to which we will want to apply Lemma \ref{lem:linking} is always $o(n)$-almost regular and so, we can indeed use the linking lemma on it. Moreover, the collection of $k$-sets $S_i$ used for the linking paths is $10k$-apart in $\pi$ and each of these sets is disjoint to $V(\mathcal{B})$ and to the sets of first and last $k$ vertices of $C_A$. By Lemma \ref{lem:destroying powers}, this means that when we remove the vertices used in the linking structures (i.e., the vertices in the sets $S_i$) from $C_A$, we are still left with the $k$-th power of a path. 

After doing the same procedure inside $B$, we have produced the desired links which, as a result, can be used to construct the $k$-th power of a Hamilton cycle in $T$.
\end{proof} 

\section{Lower bound constructions}\label{sec:lower bound cubes}
In this section, we construct tournaments with large minimum semidegree which do not contain the $k$-th power of a Hamilton cycle. Specifically, we will prove Theorems \ref{generallowerbound} and \ref{cubesthm}.
\subsection{A Construction for all $k$}
In order to prove Theorem \ref{generallowerbound}, we first will need to reduce an extremal $K_{r,r}$-free graph to a $K_{r,r}$-free graph with large minimum degree. This is done by the following standard claim.
\begin{claim}\label{proposition}
Let $r = \lceil \frac{k-1}{2} \rceil$. Then, for all $n$ there exists a balanced bipartite graph on $n-1$ vertices which is $K_{r,r}$-free, has minimum degree $\Theta \left(\frac{ ex(n,K_{r,r})}{n} \right)$ and maximum degree at most $\frac{n-1}{8}$.
\end{claim}
\noindent Now we can use this to give the tournament construction.
\begin{proof}[ of Theorem \ref{generallowerbound}]
Take $n = 3$ (mod $4$) to be sufficiently large. Let $G$ be the balanced bipartite graph on $n-1$ vertices with parts $A$, $B$ of sizes $\frac{n-1}{2}$ given by Claim \ref{proposition}, which has minimum degree $d = \Theta \left(\frac{ ex(n,K_{r,r})}{n} \right)$ and maximum degree at most $\frac{n-1}{8}$. Construct a tournament $T'$ on the same vertex set as $G$ by first orienting the edges inside $A$ and $B$ so that $T'[A]$, $T'[B]$ are regular tournaments, thus having minimum semidegree $\frac{n-3}{4}$. Secondly, orient all the edges in $G$ from $A$ to $B$ and the rest is oriented from $B$ to $A$. To finish, add a single vertex $v$ to $T'$ which has $A$ as its in-neighborhood and $B$ as its out-neighborhood. Let $T$ be the resulting $n$-vertex tournament. Note indeed that $\delta^0(T) \geq \min \left(\frac{n-3}{4}+1+d,\frac{3(n-1)}{8} \right) \geq
\frac{n+1}{4}+d$.

For sake of contradiction, suppose $T$ has the $k$-th power of a Hamilton cycle and consider the moment it passes through $v$. Let $S_A \subseteq A$ be the $r$ vertices coming before it and $S_B \subseteq B$ be the $r$ vertices coming after it. Then, every vertex in $S_A$ dominates every vertex in $S_B$ and thus, these sets form a $K_{r,r}$ in $G$, which is a contradiction.
\end{proof}
We can use a similar approach to complete our answer to one of the questions we raised in the introduction. Recall that Theorem \ref{thm:cut result}
implies that we can have $n_0 = \eps^{-O(k)}$ in Theorem \ref{bollobas}. To show that this exponential behaviour in $k$ is best possible, we can take the same construction as in the proof of Theorem \ref{generallowerbound} but instead of using Claim \ref{proposition}, we can take $G$ to be a balanced bipartite graph on $n-1=\eps^{-\Omega(k)}$ vertices which has no $K_{r,r}$, has minimum degree at least $\eps n$ and maximum degree at most $\frac{n-1}{2} - \eps n$.
Such $G$ can be easily constructed by considering a random bipartite graph with edge probability roughly $2\eps$.

\subsection{Cubes}\label{cubessection}
In this section, we prove Theorem \ref{cubesthm}, which gives an improved lower bound for containing the cube of a Hamilton cycle, which we denote by 
$H_n^3$. Let $d=d(n)$ be the largest possible odd number such that there exists a $d$-regular bipartite graph on $2n$ vertices which does not contain a copy of $C_4,C_6$ or $C_8$. Then, we can show the following.
\begin{prop}\label{prop:cubes}
Let $t$ be a positive odd integer. There exists a tournament on $n=2t+1$ vertices with $\delta(T)\geq \frac{t-1}{2}+d(t)$ which does not contain the cube of a Hamilton cycle.
\end{prop}
\begin{proof} 
 We define the vertex set of $T$ to be $A\cup B\cup\{v\}$ where $|A|=|B|=t$. Let $G$ be a $d(t)$-regular bipartite graph with parts $A$ and $B$ which does not contain $C_4,C_6$ or $C_8$. Orient all edges in $G$ from $A$ to $B$, and the remaining edges orient from $B$ to $A$. Let all vertices in $B$ dominate $v$, and let $v$ dominate all vertices in $A$. 
Now we orient the edges inside of $B$ so that they form an arbitrary regular tournament with degrees $\frac{t-1}{2}$. To finish the construction, we orient the edges in $A$. For each vertex $b\in B$, we orient the edges inside of $N_G(b)$ so that $N_G(b)$ induces a regular tournament. Notice that for distinct $b_1$ and $b_2$ the sets $N_G(b_1)$ and $N_G(b_2)$ intersect in at most one vertex, since otherwise there would be a copy of $C_4$ inside of $G$. Therefore this gives a well-defined orientation of these edges, meaning that we never orient the same edge twice. We call the edges in $A$ for which we just gave an orientation, type I edges.
Let $(a_1,a_2,a_3)$ be a triple of distinct vertices in $A$, such that $a_1a_2$ and $a_2a_3$ are oriented type I edges, but such that the edge $\{a_1,a_3\}$ is still not oriented. For each such triple, we orient this edge as $a_3a_1 \in E(T)$; we call oriented edges of this kind type II edges. Observe again that we have a well-defined orientation here, since for each type II edge $a_3a_1$ there exists a unique $a_2$ such that $a_1a_2$ and $a_2a_3$ are type I edges. Indeed, suppose the contrary. Then, there is another $a_2'$ such that $a_1a_2'$ and $a_2'a_3$ are type I edges. There are several cases here, and each results in a short cycle in $G$, which is a contradiction. Indeed, let $b_1,b_2,b'_1$ and $b_2'$ be the common neighbors in $G$ of the pairs $\{a_1,a_2\}$, $\{a_2,a_3\}$, $\{a_1,a_2'\}$, and $\{a_2',a_3\}$, respectively.
\begin{itemize}
    \item If $|\{b_1,b_2\}\cap\{b'_1,b_2'\}|=0$ then we get a $C_8$ in $G$.
     \item If $|\{b_1,b_2\}\cap\{b'_1,b_2'\}|=1$ then we get a $C_6$ in $G$.
      \item If $|\{b_1,b_2\}\cap\{b'_1,b_2'\}|=2$ then we get a $C_4$ in $G$.
\end{itemize}

Now we show that the directed graph induced by type II edges is regular. Let us first count the number of (type II) in-neighbors of any vertex $a_1\in A$; each in-neighbor is obtained by a directed $2$-path of type $I$ edges $a_1 \rightarrow a_2 \rightarrow a_3$, such that $\{a_1,a_3\}$ is not oriented as a type I edge, i.e. $\{a_1,a_3\}$ do not have a common neighbor in $B$. The vertex $a_1$ has $d(d-1)/2$ out-neighbors in the type I digraph, since $|N_G(a_1)|=d$, and for each vertex $b\in N_G(a_1)$, we have that $a_1$ has $(d-1)/2$ out-neighbors in the regular graph induced by $N_G(b)$. Furthermore, for each type I out-neighbor $a_2$ of $a_1$, there exist $(d-1)^2/2$ type I out-neighbors $a_3$ such that $\{a_1,a_3\}$ is not oriented, as $a_2$ is in $d-1$ distinct regular (type I) tournaments besides the one with $a_1$, and in each of those it has out-degree $(d-1)/2$. If for any such $a_3$ the edge $\{a_1,a_3\}$ was already oriented as a type I edge, it would mean that there is a common neighbor of $a_1$ and $a_3$, but then we would have a $C_6$ in $G$ together with the common neighbors of $a_1,a_2$ and $a_2,a_3$. We conclude that the type II in-degree of $a_1$ is precisely $d(d-1)^3/4$ and similarly, we get that the out-degree equals the same number.

Note that both the digraph formed by type I edges and the one formed by type II edges are regular, hence the complement of the union of these digraphs (i.e., the graph of non-oriented edges) inside $A$ is also regular and has even degree, as $|A|=t$ is odd.
It is a standard exercise to show that we can orient the remaining edges in $A$ so that the resulting digraph is regular, so the whole tournament on $A$ is regular. Finally, we get that the resulting tournament $T$ satisfies $\delta^0(T)\geq \frac{t-1}{2}+d(t)$.

Now suppose, for sake of contradiction, that $T$ contains a $H_n^3$. Consider the appearance of $v$ in this cycle, and notice that immediately after it, at least three vertices in $A$ necessarily follow. We consider the sequence $\pi$ of letters $A$ and $B$ corresponding to the sequence of vertices which follow after $v$ in $H_n^3$ and their affiliation to either $A$ or $B$. 

As noted, our sequence starts with at least three consecutive $A$'s. Look at the last two of these consecutive $A$'s  : we either have that two $B$'s follow, which would give a $C_4$ in $G$ (a contradiction), or we have a $B$ followed by an $A$. So we may assume that there is a subsequence $AABA$. Now, either we can find a subsequence $AABAB$ in our graph (and we deal with this case in the next passage), or the sequence we found continues with an $A$, giving $AABAA$. Furthermore, each time we encounter next $B$ in our sequence, we would have to continue with two consecutive $A$'s, which would mean we have more $A$'s than $B$'s in the sequence, which is a contradiction since $|A|=|B|$.

Therefore, there has to be a subsequence $AABAB$ in $H_n^3$ whose vertices we denote by $a_1, a_2, b_1, a_3, b_2$. This implies that $a_1, a_2$ and $a_3$ form a transitive tournament in $A$ with $a_1\rightarrow a_2\rightarrow a_3$, and $b_1$ is the common neighbor of $a_1,a_2$ in $G$, and $b_2$ is the common neighbor of $a_2,a_3$ in $G$. Notice that $a_1a_2$ and $a_2a_3$ are type I edges, and if we prove that $\{a_1,a_3\}$ do not have a common neighbor, then $a_3a_1$ is a type II edge, contradicting the assumption that $a_1, a_2$ and $a_3$ form a transitive tournament. If $b_2$ is the common neighbor of $a_1, a_3$ then we have a $C_4$ in $G$ formed by $a_1-b_1-a_2-b_2-a_1$. Also, $b_1$ can not be the common neighbor, since 
$b_1 \rightarrow a_3$ and therefore this edge is not $G$. Otherwise, if some other vertex $b_3$ is the common neighbor, we again get a contradiction, since we have a $C_6$ in $G$, given by $a_1-b_1-a_2-b_2-a_3-b_3-a_1$. This finishes the proof.
\end{proof}

\noindent
To complete the proof we need the well known result of Benson \cite{benson1966minimal} on Tur\'an numbers of even cycles.

\begin{prop}
For every $q=2^{\ell}$ there exists a balanced bipartite graph $H$ on $2\sum_{i=0}^5 q^i$ vertices which is $(q+1)$-regular and has girth $12$.
\end{prop}
 This result shows that for $q=2^{\ell}$ and $t=\sum_{i=0}^5q^i$, it holds that $d(t)\geq q+1=\Omega(t^{1/5})$. By considering $n=2\sum_{i=0}^5q^i+1$ and using Proposition \ref{prop:cubes}, we obtain Theorem \ref{cubesthm}.

\section{Concluding remarks}\label{sec:concluding remarks}    
In this paper, we resolved, for all even $k$, the question of determining the minimum semidegree condition which ensures that a tournament contains the $k$-th power of a Hamilton cycle. For odd $k$, although we made a very significant improvement on what was previously known, there is still a small gap between our bounds and it would be very interesting to close it. Here, the first open case is $k=3$. In this case, our Theorem \ref{cubesthm} gives a lower bound of order $n^{1/5}$
on the additive error term. It might be possible to improve this result using similar ideas as before, by forbidding in the constructions subgraphs whose Tur\'an number is larger than that of $C_8$. On the other hand, the upper bound in this case,
coming from Theorem \ref{maintheorem}, has order $\sqrt{n}$ and it is not even clear what the truth should be.

\bibliographystyle{plain}

\end{document}